\documentclass[11pt]{amsart}
\usepackage{amsmath}
\usepackage{amscd}
\usepackage{amssymb}
\usepackage{amsthm}
\usepackage{dsfont} 
\usepackage{appendix} 
\usepackage{verbatim}
\usepackage[normalem]{ulem} 
\usepackage[all]{xy} 
\usepackage{enumerate}
\usepackage{stackrel}
\usepackage{color}
\input cyracc.def
\DeclareFontFamily{U}{russian}{}
\DeclareFontShape{U}{russian}{m}{n}
        { <5><6> wncyr5
        <7><8><9> wncyr7
        <10><10.95><12><14.4><17.28><20.74><24.88> wncyr10 }{}
\DeclareSymbolFont{Russian}{U}{russian}{m}{n}
\DeclareSymbolFontAlphabet{\mathcyr}{Russian}
\makeatletter
\let\@math@cyr\mathcyr
\renewcommand{\mathcyr}[1]{\@math@cyr{\cyracc #1}}
\makeatother

\begin{document}


\theoremstyle{plain}
\newtheorem{theorem}{Theorem}[section]
\newtheorem{introtheorem}{Theorem}

\theoremstyle{plain}
\newtheorem{proposition}[theorem]{Proposition}
\newtheorem{prop}[theorem]{Proposition}

\theoremstyle{plain}
\newtheorem{corollary}[theorem]{Corollary}

\theoremstyle{plain}
\newtheorem{lemma}[theorem]{Lemma}

\theoremstyle{plain}
\newtheorem{expectation}[theorem]{Expectation}

\theoremstyle{plain}
\newtheorem{definition}[theorem]{Definition}

\theoremstyle{remark}
\newtheorem{remind}[theorem]{Reminder}

\theoremstyle{definition}
\newtheorem{condition}[theorem]{Condition}

\theoremstyle{definition}
\newtheorem{construction}[theorem]{Construction}

\theoremstyle{definition}
\newtheorem{example}[theorem]{Example}

\theoremstyle{definition}
\newtheorem{notation}[theorem]{Notation}

\theoremstyle{definition}
\newtheorem{convention}[theorem]{Convention}

\theoremstyle{definition}
\newtheorem{assumption}[theorem]{Assumption}
\newtheorem{isoassumption}[theorem]{Isomorphism Assumption}

\newtheorem{finassumption}[theorem]{Finiteness Assumption}
\newtheorem{indhypothesis}[theorem]{Inductive Hypothesis}

\theoremstyle{remark}
\newtheorem{remark}[theorem]{Remark}

\numberwithin{equation}{subsection}





\newcommand{\mathscr}[1]{\mathcal{#1}}
\newcommand{\TODO}{{\color{red} TODO}}
\newcommand{\CHANGE}{{\color{red} CHANGE}}
\newcommand{\annette}[1]{{\color{red}Annette: #1 }}
\newcommand{\simon}[1]{{\color{green}Simon: #1 }}
\newcommand{\giuseppe}[1]{{\color{blue}Giuseppe: #1 }}

\newcommand{\tensor}{\otimes}
\newcommand{\Fil}{\tn{Fil}}
\newcommand{\Gh}{\mathcal{G}}
\newcommand{\Fh}{\shfF}
\newcommand{\Oh}{\mathcal{O}}
\newcommand{\GrVec}{\textnormal{GrVec}}
\newcommand{\BN}{\textnormal{BN}}
\newcommand{\modulo}{\textnormal{mod}}
\newcommand{\CH}{\mathrm{CH}}
\newcommand{\id}{\mathrm{id}}
\newcommand{\isocan}{\xrightarrow{\hspace{1.85pt}\sim \hspace{1.85pt}}}
\newcommand{\isom}{\cong}
\newcommand{\red}{\mathrm{red}}
\newcommand{\Betti}{R_B}
\newcommand{\ladic}{R_\ell}
\newcommand{\Hodge}{R_H}
\newcommand{\MHM}{\mathrm{MHM}}
\newcommand{\Hh}{\mathcal{H}}
\newcommand{\A}{\mathbb{A}}
\newcommand{\gm}{\mathrm{gm}}
\newcommand{\qfh}{\mathrm{qfh}}
\newcommand{\DM}{\tn{\tbf{DM}}}
\newcommand{\DA}{\tn{\tbf{DA}}}
\newcommand{\DMeff}{\tn{\tbf{DM}}^{\eff}}
\newcommand{\DAeff}{\tn{\tbf{DA}}^{\eff}}
\newcommand{\DAgm}{\tn{\tbf{DA}}_{\gm}}
\newcommand{\Bei}{\mathcyr{B}}
\newcommand{\kd}{\tn{kd}}
\newcommand{\Sm}{\tn{\tbf{Sm}}}
\newcommand{\Mor}{\textnormal{Mor}}
\newcommand{\Hom}{\textnormal{Hom}}
\newcommand{\Sh}{\tn{\tbf{Sh}}}
\newcommand{\et}{\tn{\'{e}t}}
\newcommand{\an}{\tn{an}}
\newcommand{\D}{\tn{D}}
\newcommand{\eff}{\tn{eff}}
\newcommand{\DMgm}{\DM_{\tn{gm}}}
\newcommand{\vp}{\varphi}
\newcommand{\Sym}{\textnormal{Sym}}
\newcommand{\OSym}{\textnormal{coSym}}
\newcommand{\Mof}{M_1}
\newcommand{\CGS}{\tn{\tbf{cGrp}}}
\newcommand{\tr}{\textnormal{tr}}
\newcommand{\one}{\mathds{1}}
\newcommand{\Spec}{\textnormal{Spec}}
\newcommand{\Sch}{\tn{\tbf{Sch}}}
\newcommand{\Nor}{\tn{\tbf{Nor}}}
\newcommand{\Lie}{\tn{Lie}}
\newcommand{\Ker}{\tn{Ker}}
\newcommand{\Frob}{\tn{Frob}}
\newcommand{\Ver}{\tn{Ver}}
\newcommand{\N}{\mathbb{N}}
\newcommand{\Z}{\mathbb{Z}}
\newcommand{\Q}{\mathbb{Q}}
\newcommand{\Ql}{\mathbb{Q}_{\ell}}
\newcommand{\R}{\mathbb{R}}
\newcommand{\C}{\mathbb{C}}
\newcommand{\G}{\mathbb{G}}

\newcommand{\xra}{\xrightarrow}
\newcommand{\xla}{\xleftarrow}
\newcommand{\sxra}[1]{\xra{#1}}
\newcommand{\sxla}[1]{\xla{#1}}

\newcommand{\sra}[1]{\stackrel{#1}{\ra}}
\newcommand{\sla}[1]{\stackrel{#1}{\la}}
\newcommand{\slra}[1]{\stackrel{#1}{\lra}}
\newcommand{\sllra}[1]{\stackrel{#1}{\llra}}
\newcommand{\slla}[1]{\stackrel{#1}{\lla}}

\newcommand{\ira}{\stackrel{\simeq}{\ra}}
\newcommand{\ila}{\stackrel{\simeq}{\la}}
\newcommand{\ilra}{\stackrel{\simeq}{\lra}}
\newcommand{\illa}{\stackrel{\simeq}{\lla}}

\newcommand{\ul}[1]{\underline{#1}}
\newcommand{\tn}[1]{\textnormal{#1}}
\newcommand{\tbf}[1]{\textbf{#1}}

\newcommand{\Spt}{\mathop{\mathbf{Spt}}\nolimits}
\newcommand{\ra}{\rightarrow}
\newcommand{\old}{\mathrm{old}}

\setcounter{tocdepth}{1}

\title{On the relative motive of a commutative group scheme}

\author{Giuseppe Ancona}
\address{Universit\"at Duisburg-Essen, Thea-Leymann-Stra\ss e 9 ,
45127 Essen, Germany}

\email{monsieur.beppe@gmail.com}
\author{Annette Huber}
\address{Mathematisches Institut, Albert-Ludwigs-Universit\"at Freiburg, Eckerstra\ss e~1,  
79104 Freiburg im Breisgau, Germany}
\email{annette.huber@math.uni-freiburg.de}

\author{Simon Pepin Lehalleur}
\address{Instit\"ut f\"ur Mathematik, Universit\"at Z\"urich, Switzerland}
\email{simon.pepin@math.uzh.ch}

\begin{abstract}
We prove a canonical K\"unneth decomposition of the relative motive with rational coefficients of a smooth commutative group scheme over a noetherian finite dimensional base.
\end{abstract}

\maketitle
\begin{center}
\today
\end{center}
\tableofcontents

\section*{Introduction}

Let $S$ be a noetherian scheme of finite dimension and $G$ a smooth commutative $S$-group scheme of finite type. The main objective of this paper is to 
establish  
the following canonical K\"unneth decomposition of the motive $M_S(G)$ associated to $G$ in the triangulated category of
 motives over $S$ with rational coefficients (see Theorems \ref{MainThm} and \ref{MainThmwithtransfers}):
\[M_S(G)\isocan \left(\bigoplus_{n\geq 0}^{\kd(G/S)} \Sym^n\Mof(G/S)\right) \otimes M(\pi_0(G/S)) \ .\tag{*}\]
In the above, $\Mof(G/S)$ is the rational $1$-motive of $G$, i.e., the motive induced by the \'etale sheaf represented by $G\tensor\Q$, the motive $M(\pi_0(G/S))$ is a fibrewise Artin motive related to the connected components of fibres of $G$ and $\kd(G/S)$ is a non-negative integer. This decomposition is natural in $G$, it respects the structure of Hopf objects and it is compatible
with pullback for morphisms $f:T\to S$. 

This level of generality makes it possible to treat examples like degenerating families of abelian or semi-abelian varieties, in particular N\'eron models or the universal family over toroidal compactifications of mixed Shimura varieties of Hodge type. Note also that there are no assumptions on regularity or the residual characteristics of $S$. Moreover, we show that this K\"unneth decomposition also holds when $G$ is a commutative group in the category of smooth algebraic spaces over $S$. However, the argument uses rational coefficients in many places. One does not expect such a direct sum decomposition with integral coefficients, even when $S$ is the spectrum of a field.

\

This result has several interesting consequences. First, the motive $\Mof(G/S)$ is geometric and the motives $M_S(G)$ and $\Mof(G/S)$ are of finite Kimura dimension. Second, the Chow groups of $G$ with rational coefficients decomposes into a finite sum of eigenspaces with respect to the multiplication by $n$ (Theorem~\ref{corchow}); this generalizes the result of Beauville for abelian varieties \cite{Beau}. Third, it is now possible to construct the motivic polylog on
$G\smallsetminus\{e_G\}$ for all $S$ and $G$ as above (with $e_G$ the zero section). This is a tremendous generalization
from the cases $G=\G_m$ or $G$ an abelian scheme over a regular base available before.
At the same time, the systematic use of motives over $G$ simplifies the construction
even in the classical case; for details see the upcoming paper \cite{HK14}.

\

The motive $\Mof(G/S)$ is a prime example of a relative homological $1$-motive, i.e., a motive in the triangulated subcategory generated by the relative homology of curves over $S$. In fact, the results of Section \ref{section:real} below suggest that $\Mof(G/S)[-1]$ is in the heart of the conjectural standard motivic $t$-structure on homological $1$-motives. The upcoming thesis \cite{thesis-simon} will make this statement rigorous.

\

Some special cases of the main theorem are already known. The case of abelian schemes over a regular base was treated in the
setting of Chow motives by K\"unnemann in \cite{Ku1}. Building on this, the case where $S$ is the spectrum of a perfect field and general $G$ was treated by the first
two authors together with Enright-Ward in \cite{AEH}. We refer to the introduction of \cite{AEH} for a more complete discussion of earlier work.

\ 

Indeed, the present note is
a follow-up to \cite{AEH}. The idea of the proof is simple: once the morphism $(*)$ is written down, we
can check that it is an isomorphism after pullback to geometric points of
$S$, where we are back in the perfect field case. 

However, a number of difficulties
arise. Firstly, there are technical advantages to working in the setting
of motives without transfers (see Remark~\ref{Harrer}) and it is in this setting that we
define the morphism $(*)$. 
So we will be forced to  compare to the approach of \cite{AEH} which uses Voevodsky's category of motives with transfers.
Secondly, the sheaf represented by $G$ is not cofibrant for the
$\A^1$-model structure, hence it is not obvious how to compute
the left derived pullback $f^*\Mof(G/S)$. This is overcome by an explicit
resolution defined by rational homotopy theory. Finally, the crucial reduction that allows to check that a morphism is an isomorphism after restriction to geometric fibres (Lemma~\ref{vanishing}) is only available in the stable case. 

All these technical difficulties give rise to results that have their own interest. In particular we show that the sheaf represented by $G\otimes \Q$ admits transfers (Theorem \ref{G with transfers}).

\ 

We will (mostly) work with stable categories of motives with rational coefficients. There are several version of such categories available in the literature, namely: \begin{enumerate}
\item the original definition of Voevodsky of $\DM(S,\Q)$ based on finite correspondences and sheaves with transfers,
\item the \'etale motives without transfers $\DA^\et(S,\Q)$, introduced by Morel and studied in \cite{Ayoub_these_1}, \cite{Ayoub_these_2}, \cite{Ayoubet},
\item h-motives $\DM_h(S,\Q)$, originally introduced by Voevodsky and studied in \cite{CD, CDet},
\item Beilinson motives $\DM_{\Bei}(S)$ introduced in \cite{CD}.
\end{enumerate} 
All these categories are related by functors, and moreover are equivalent for good base schemes (e.g. $S$ regular) by \cite[Theorems 16.1.2, 16.1.4, 16.2.18]{CD}.

We have several reasons here to work with $\DA^\et(S,\Q)$; for instance, the morphism $(*)$ is easiest to write in this context (again, see Remark \ref{Harrer}). The comparison functors mentioned above allow us to transpose the main theorem to the other categories (see Theorem \ref{MainThmwithtransfers} and Remark \ref{other categories}).
\subsection*{Organization of the paper} Section \ref{sect1} deals with generalities on commutative group schemes. Section~\ref{sect2} introduces the motive $\Mof(G/S)$ and gives its basic properties. Section~\ref{sect3} states the main theorem and its consequences. The proof is then given in Section \ref{sect:proof}. We study the Betti, $\ell$-adic and Hodge realizations of the K\"unneth components in Section \ref{section:real}. In Section~\ref{sect:alggroups} we present the straightforward generalisation of our results to the setting of smooth commutative algebraic group spaces. 

The reader who is not familiar with the triangulated categories of mixed motives over a base \cite{Ayoubet,CD} may find Appendix \ref{app_relative} useful. The  three other appendices deal with technical points of the proof of the main theorem explained above. In Appendix \ref{app_transfers} we develop the tools needed to translate the main result of \cite{AEH} to motives without transfers.  In Appendix~\ref{app_qfh} we establish $\qfh$-descent for the presheaf given by a commutative group. Appendix \ref{app_resolution} deals with the construction of a functorial cofibrant resolution for the sheaf $G\tensor\Q$. 
\subsection*{Ackowledgments} It is a pleasure to thank Joseph Ayoub, Fr\'ed\'eric D\'eglise, Brad Drew,  Michael Rapoport and Fabio Tonini for discussions. The first two authors are particularly thankful to Fr\'ed\'eric D\'eglise for pointing out the serious difficulty with pullback of $\Mof(G)$.  We are also very grateful to 
Joseph Ayoub for pointing out a mistake in an earlier version of Appendix~\ref{app_transfers}.
This work is part of the third author's PhD thesis at the University of Z\"urich under the supervision of Joseph Ayoub. We learned a lot about algebraic group spaces from reading up
a 2009 discussion on MathOverflow. We thank the participants.

\section*{Conventions}
Throughout let $S$ be a noetherian scheme of finite dimension. Let $\Sm/S$ be the category of smooth $S$-schemes of finite type and $\Sch/S$ the category of $S$-schemes of finite type. We also write $\Sm$ and $\Sch$ when there is no ambiguity on the base. By sheaf we mean an \'etale sheaf on $\Sm/S$, unless specified otherwise.
By $G$, we denote either a group scheme or an algebraic group space that is always assumed to be commutative, smooth of finite type over $S$. In this situation, we often write $G/S$ for concision. For any morphism $f: T \longrightarrow S$ we will write \[G_T/T=G\times_S T / T\] for the base change of $G/S$ along $f$. We write $\CGS_S$ for the category of smooth commutative group schemes of finite type over $S$.

Throughout, we write $\ul{G/S}$ or more briefly $\ul{G}$ for the presheaf of abelian groups on $\Sch/S$ or $\Sm/S$ defined by  $G$. We write $\ul{G/S}_\Q$ (or $\ul{G}_\Q$) for
the presheaf tensor product $\ul{G/S}\tensor_\Z\Q$. 

We use the cohomological indexing convention for complexes, except in Appendix \ref{app_resolution} where we follow the homological convention of \cite{Breen} for ease of reference.

As discussed in the introduction, we need to consider various triangulated categories of mixed motives with rational coefficients. See the Appendices \ref{app_relative} and \ref{app_transfers} for the set-up we use. In particular, we write $\DA(S)$ for $\DA^\et(S,\Q)$.

At several points, we have parallel statements in effective and stable categories of motives. They are indicated by an $(\eff)$ in the notation.

\section{Group schemes}\label{sect1}

Let $S$ be a noetherian scheme of finite dimension and $G/S$ a smooth commutative group scheme of finite type.

Since $G/S$ is smooth, the union of the neutral components $G_s^0$ for $s\in S$ is open in $G$ by \cite[Expos\'e VI$_B$ Th\'eor\`eme 3.10]{SGA3}. We write $G^0$ for the corresponding open subscheme of $G$ and $\pi_0(G/S):=G/G^0$ for the quotient \'etale sheaf on $\Sch/S$. We also write $\pi_0(G/S)$ for its restriction to $\Sm/S$. By definition of $G^0$, the formation of $G^0$ and $\pi_0(G/S)$ for $G/S$ smooth commutes with base change.

The following statement uses the notion of an algebraic space. For basics on algebraic spaces, see Section \ref{sect:alggroups} below.

\begin{lemma}
\label{lemma:connected_comp}
The scheme $G^0$ is a smooth $S$-group scheme of finite type and $\pi_0(G/S)$ is an \'etale algebraic group space. If $G^0$ is closed in $G$ (for instance if $S$ is the spectrum of a field) then $\pi_0(G/S)$ is a scheme. Moreover, there is a canonical
short exact sequence of \'etale sheaves
\[0 \longrightarrow G^0 \longrightarrow G \stackrel{p_G}{\longrightarrow} \pi_0(G/S) \longrightarrow 0 \ .\]
\end{lemma}

\begin{proof}
Since $G/S$ is smooth of finite type and $S$ is noetherian, $G^0$ is also smooth and of finite type. 

As the quotient of the smooth equivalence relation $G^0\times_G(G\times_SG)$, the sheaf $G/G^0$ is an $S$-algebraic group space (see e.g \cite[Corollaire 10.4]{LMB}).
 
The algebraic group space $\pi_0(G/S)$ is locally of finite presentation (because both $G$ and $G^0$ are locally of finite presentation) and formally \'etale (let $T^0\rightarrow T$ be an infinitesimal thickening; then $(G/G^0)(T)\rightarrow (G/G^0)(T^0)$ is surjective because $G$ is smooth, and injective because $G^0$ is open in $G$). We conclude that $\pi_0(G/S)$ is an \'etale algebraic space. 

 If $G^0$ is open and closed in $G$ then $\pi_0(G/S)$ is separated. A quasi-finite separated algebraic space of finite presentation is a scheme by \cite[II, Theorem 6.15]{Knutson}.
\end{proof}

The function $s\mapsto |\pi_0(G_{\bar{s}}/\bar{s})|$ (where $\bar{s}$ is any geometric point above $s$) is locally constructible on $S$ by \cite[Corollaire 9.7.9]{EGAIV_3}. In particular, it is bounded since $S$ is quasi-compact. This justifies the following definition.

\begin{definition}\label{orderpi}
The {\em order} of $\pi_0(G/S)$, denoted $o(\pi_0(G/S))$, is defined as the least common multiple of the order of all the elements of the groups $\pi_0(G_{\bar{s}}/\bar{s})$ (with $\bar{s}$ geometric point of $S$). 
\end{definition}

\begin{definition}\label{globalKimura}
For any point $s \in S$, we write $G_s$ for the fibre, $g_s$ for its abelian rank and $r_s$ for its torus rank.
 The {\em Kimura dimension} of $G/S$ is:
\[\kd(G/S):=\max\{2g_s+r_s|s\in S\}\]
\end{definition}
This terminology with be justified by Theorem \ref{MainThm}.

\begin{lemma}\label{lemma:kd_semi_cont}
The value $\kd(G/S)$ is the maximum of $2g_s+r_s$ for $s$ varying in generic points of $S$.
\end{lemma}
\begin{proof}
After replacing $G$ by $G^0$, we can suppose that $G$ is fibrewise connected. Let us fix $t \in S$. After base changing to the strict henselianisation of the local ring at $t$, we can suppose $S$ to be strict henselian. It is enough to show that, under this hypothesis, $2g_s+r_s\geq 2g_{t}+r_{t} $ for all $s\in S$.

Let us fix a prime $\ell$ which is coprime to the residual characteristic of $S$ and, for all natural $n$, consider the multiplication map $[\ell^n]:G \rightarrow G$. The integer $2g_s+r_s$ is the rank of the $\ell$-adic Tate module of $G_s$. By \cite[\S 7.3 Lemma 2(b)]{Neron}, the group scheme $\ker [\ell^n]$ is \'etale over $S$. So, by Hensel's Lemma, any section of this \'etale morphism at $t$ extends to a section over $S$. In particular the rank of the Tate module of $G_s$ has its minimum at $s=t$. 
\end{proof}

\section{The $1$-motive $\Mof(G)$}\label{sect2}

\begin{definition}\label{def:sheaf_G}
 Let $\ul{G/S}$ be the \'etale sheaf of abelian groups on $\Sm/S$ defined by $G$,
i.e.,
\[ \ul{G/S}(Y)=\Mor_{\Sm/S}(Y,G)\]
for $Y\in \Sm/S$. 

Let \[a_{G/S}: \Z\Mor_S(\cdot, G) \rightarrow \ul{G/S}\] 
be the morphism of presheaves of abelian groups on $\Sm/S$ induced by the addition map 

We may omit $S$ from these notations if the base scheme is clear.
\end{definition}

\begin{remark}
By \'etale descent, $\ul{G/S}$ is a sheaf; $\ul{G/S}_\Q$ is then also a sheaf by a quasi-compactness argument \cite[Lemma 2.1.2]{AEH}. However, the presheaves $\Z\Mor_S(\cdot, G)$ and $\Q\Mor_S(\cdot, G)$ are not sheaves.
\end{remark}

\begin{definition} 
\label{def:main_defM1}
\begin{enumerate}
\item Define
 \[\Mof^\eff(G/S)\in \DAeff(S)\]
 to be the motive induced by the sheaf $\ul{G/S}_\Q$ and
\[\Mof(G/S)=L\Sigma^\infty\Mof^\eff(G/S)\in \DA(S) \ .\]
 \item Define
 \[\alpha^\eff_{G/S}: M^\eff_S(G) \longrightarrow \Mof^\eff(G/S)\] 
 to be the morphism in $\DA(S)$ induced by the sheafification of ${a_{G/S}\otimes\Q}$ and
\[\alpha_{G/S}=L\Sigma^\infty\alpha^\eff_{G/S}: M_S(G)\longrightarrow \Mof(G/S)\]
\end{enumerate}
\end{definition}

\begin{remark}\label{Harrer}
The reader should compare with \cite[Definition 2.1.5]{AEH} which is an effective analogue with transfers over a perfect field. A comparison between the two morphisms will be made in Section \ref{subsect:field}. The definition in op. cit. cannot be generalized over a general $S$, because it is not yet known that finite correspondences from $X$ to $Y$ are related to morphisms from $X$ to a symmetric power of $Y$. This point is studied in the upcoming work of Harrer \cite{thesis-Harrer}.
\end{remark} 

\begin{remark}
The assignment
\[G/S \mapsto \Mof(G/S)\]
naturally extends to a functor \[\Mof: \CGS_S \longrightarrow \DA(S) \ .\]
\end{remark}

\begin{lemma}
\label{lemma:connected_reduction}
\begin{enumerate}
\item \label{M1linear} The functor $\Mof$ is $\Z$-linear. Moreover, it sends short exact sequences in $\CGS_S$ to exact triangles in $\DA(S)$.
\item The inclusion $G^0\rightarrow G$ induces an isomorphism $\Mof(G)\simeq\Mof(G^0)$.

\item \label{alphalinear} The morphism $\alpha_{G/S}$ is natural in $G \in \CGS_S.$
\end{enumerate}
\end{lemma}

\begin{proof}
\begin{enumerate}
\item The functor from $\CGS_S$ to the category of \'etale sheaves of $\Q$-vector spaces on $\Sm/S$ which sends $G$ to $\ul{G/S}_\Q$ is $\Z$-linear and exact. The statement then follows from the construction of $\DA(S)$.

\item As the quotient $G/G^0$ is a torsion sheaf the map $G^0\ra G$ induces an isomorphism of sheaves ${\ul{G^0/S}_\Q\stackrel{\sim}{\rightarrow}\ul{G/S}_\Q}$ . 
\item The morphism $a_{G/S}\otimes \Q$ is natural in $G \in \CGS_S,$ and so is its sheafification.
\end{enumerate}
\end{proof}

\begin{proposition}\label{pullback_group}
Let $f:T\to S$ be a morphism and $G\to S$ a smooth commutative group scheme. Then with the notation of Definition \ref{def:main_defM1} there are canonical isomorphisms
\begin{gather*} 
f^*\Mof^{(\eff)}(G/S)\simeq\Mof^{(\eff)}(G_T/T)\\
\end{gather*}
and, modulo these isomorphisms, we have equalities of morphisms
\begin{gather*}
f^*(\alpha^{(\eff)}_{G/S})=\alpha^{(\eff)}_{G_T/T} \ .
\end{gather*}
\end{proposition}

\begin{proof}
Since pullbacks commute with suspension (Lemma \ref{lemma_motives}(\ref{pullback_suspension})) it is enough to treat the effective case. In this proof, we write for clarity $f^*$ for the underived pullback on complexes of sheaves and $Lf^*:\DAeff(S)\rightarrow \DAeff(T)$ for the triangulated pullback functor.

By the universal property of fibre products, we have
\[ f^*\ul{G}=\ul{G\times_ST}\ .\]
By applying Theorem \ref{theorem:resolution} to the sheaf $\ul{G}$ and switching to cohomological indexing we obtain a complex $A(\ul{G})\in\mathbf{Cpl}^{\leq 0}(\Sh_\et(\Sm/S,\Z)$ 
together with a map
\[
r:A(\ul{G})\rightarrow \ul{G}
\]
with the following properties.
\begin{enumerate}
\item \label{eform} For all $i\in \N$, the sheaf $A(\ul{G})^{-i}$ is of the form $\bigoplus_{j=0}^{d(i)}\Z(\Gh^{a(i,j)})$ for some $d(i),a(i,j)\in\N$. In particular $A_\Q(G):=A(\ul{G})\tensor\Q$ is a bounded above chain complex of sums of representable sheaves of $\Q$-vector spaces; hence it is cofibrant in the projective model structure that we use to define $\DAeff(S)$.
\item \label{lift} There is a morphism $\tilde{a}_{G/S}:\Z(G)\rightarrow A(\ul{G})$ lifting $a_{G/S}$.
\item \label{Qres} The map $r_\Q:=r\tensor \Q$ is a quasi-isomorphism.
\end{enumerate}

There is a complex $A(\ul{G_T})\in\mathbf{Cpl}^{\leq 0}(\Sh_\et(\Sm/T,\Z))$ and a map $r_T:A(\ul{G_T})\rightarrow \ul{G_T}$ with the same properties, and all those objects are compatible with pullbacks. Putting all of this together, we obtain canonical isomorphisms of objects

\begin{equation*}
Lf^*M_1(G/S)\stackrel[\sim]{Lf^*(r_\Q)}{\xleftarrow{\hspace*{1cm}}} f^*(A_\Q(G))\simeq A_\Q(G_T)\stackrel[\sim]{r_\Q}{\xrightarrow{\hspace*{1cm}}} M_1(G_T/T) \\
\end{equation*}

and, modulo these isomorphisms, equalities of morphisms:

\begin{equation*}
Lf^*\alpha^\eff_{G/S}=f^*(\tilde{a}_{G/S}\otimes\Q)=\tilde{a}_{G_T/T}\otimes\Q=\alpha^\eff_{G_T/T}
\end{equation*}

This finishes the proof.
\end{proof}

The rest of the section is devoted to study the $1$-motive \textit{with transfers} of $G$ and to compare it to $\Mof(G)$. It will be used only in section \ref{sect:proof} and can be skipped on a first reading.

\begin{theorem}\label{G with transfers}
Let $G$ be a smooth commutative group scheme over an excellent (noetherian and finite dimensional) scheme $S$. The \'etale sheaf $\ul{G/S}_\Q$  on $\Sm/S$ represented by $G$ has a unique structure of \'etale sheaf with transfers, which we denote $\ul{G/S}_\Q^\tr$.
Moreover, there is a unique map
\[ a_{G/S}^\tr:\Q(G)^\tr\to \ul{G/S}_\Q^\tr\]
of sheaves of transfers extending $a_{G/S}$.
\end{theorem} 
\begin{proof}
This comes from the fact that $\ul{G/S}_\Q$ is a $\qfh$-sheaf (Proposition \ref{prop:qfh-descent}) and that such sheaves have a unique structure of sheaves with transfers (as follows from Theorem \ref{SVqfh} and the Yoneda lemma).
The morphism $a_{G/S}^\tr$ is defined by taking the $\qfh$-sheafification of $a_{G/S}$ and using the natural isomorphisms of Theorem \ref{SVqfh} and Proposition \ref{prop:qfh-descent}. Uniqueness follows again from Yoneda.
\end{proof}

We can then proceed as in the case of $\Mof(G)$.

\begin{definition}\label{M1 with transfers} 
\begin{enumerate}
\item
Define the (effective) $1$-motive of $G$ 
 \[\Mof^{(\eff)}(G/S)^{\tr}\in \DM^{(\eff)}(S)\]
 to be the motive induced by the sheaf $\ul{G/S}_\Q$ (analogously to Definition 
\item  Define
 \[\alpha^\tr_{G/S}: M_S(G)^\tr \longrightarrow \Mof(G/S)^\tr\] 
 to be the morphism in $\DM(S)$ induced by ${a_{G/S}^\tr}$.
\end{enumerate}
\end{definition}

Recall the adjoint pairs of functors  (see \S \ref{subsection sheaves with without} and \S \ref{subsection with without})
\[ \gamma^*:\Sh_\et(\Sm)\rightleftarrows \Sh_\et(\Sm^\tr):\gamma_* \]
and
\[ L\gamma^*:\DAeff(S)\rightleftarrows \DMeff(S):\gamma_*\ .\]

\begin{proposition}\label{prop: with without}
Let $S$ be excellent.
Let $G$ be a smooth commutative group scheme over $S$.
\begin{enumerate}
\item \label{item: with without 1}
The natural morphism induced from the counit of $\gamma^*\dashv\gamma_*$
for \'etale sheaves
\[  \gamma^*\gamma_*\ul{G/S}^\tr_\Q\to\ul{G/S}^\tr_\Q\]
is an isomorphism of sheaves with transfers. 
\item\label{item: with without 2}
 The natural morphism (induced from the counit of $L\gamma^*\dashv \gamma_*$ for effective motives, followed by infinite suspension in the stable case)
\[L\gamma^*\Mof^{(\eff)}(G/S)\rightarrow\Mof^{(\eff)}(G/S)^\tr\]
is an isomorphism.
\end{enumerate}
\end{proposition}

\begin{proof}
We have
\[ L\gamma^*\Mof(G)=L\gamma^*L\Sigma^\infty\Mof^\eff(G)\simeq L\Sigma^\infty_\tr L\gamma^*\Mof^\eff(G)\]
 by Lemma \ref{lemma_transfers} (\ref{suspension_transfers}). Since this isomorphism is compatible with the counit maps of the adjunctions $L\gamma^*\dashv \gamma_*$, it is thus enough to treat the effective case in (\ref{item: with without 2}). We first reduce to item (\ref{item: with without 1}) as follows.

 We apply Theorem \ref{theorem:resolution} to the sheaf on $\Sch/S$ represented by $G$ 
(that we also denote by $\ul{G}$ for simplicity). This yields a resolution
\[ r_\Q:A_\Q(G)\longrightarrow \gamma_*(\ul{G}_\Q)^\tr\]
where $A_\Q(G)$ is a bounded above complex of sheaves of $\Q$-vector spaces
whose terms are finite sums of sheaves represented by smooth $S$-schemes. Hence
$(r_\Q)_{|\Sm}$ is a cofibrant resolution of $\ul{G}_\Q$ in the projective model structure on complexes of sheaves on $\Sm/S$. 
By definition of the derived counit of a Quillen adjunction, the counit map $\eta_{\Mof^\eff(G/S)^\tr}$ is the image of the following composition at the model category level:
\[
\gamma^*A_\Q(G)_{|\Sm}\stackrel{\gamma^*(r_Q)_{\Sm}}{\longrightarrow} \gamma^*\gamma_*(\ul{G}_\Q)^\tr\stackrel{\eta_{(\ul{G}_\Q)^\tr}}{\longrightarrow}(\ul{G}_\Q)^\tr \ .
\] 
The assumptions of Lemma \ref{specialgamma} are satisfied for $A_\Q\stackrel{r_\Q}{\rightarrow}{\ul{G}_\Q}$. By point (\ref{exact_gamma}) of this lemma, the first map in this composition is a quasi-isomorphism. 
It thus remains to show that the second map is an isomorphism, which is precisely item (\ref{item: with without 1}). 

The functor $\gamma_*$ at the level of sheaves with transfers is conservative so it is enough to show that $\gamma_*\eta_{(\ul{G}_\Q)^\tr}$ is an isomorphism. By Lemma \ref{specialgamma} (\ref{gamma_qfh}) and the triangular identity of adjunctions we have the following commutative diagram:
\[
\xymatrix{\gamma_*\gamma^*\ul{G}_\Q \ar@/^1pc/[rr]^{\id}\ar[r]_{\gamma_*\eta} & \ul{G}_\Q \ar[r]_\epsilon \ar[rd]_{\epsilon'} & \gamma_*\gamma^*\ul{G}_\Q\\
& & (\ul{G}_\Q)_{\qfh}|_\Sm \ar[u]_{\sim}.
}
\]
Proposition \ref{prop:qfh-descent} shows that the map $\epsilon'$ is an isomorphism. Together with the diagram, this finishes the proof.
\end{proof}

\begin{proposition}\label{prop comp excellent}
Let $S$ be excellent. 
Modulo the isomorphisms of Proposition \ref{prop: with without}, 
$\gamma^*$ sends the morphism $a_{G/S}$ to $a_{G/S}^\tr$ and
$L\gamma^*$ sends the morphism
$\alpha_G$ to $\alpha_G^\tr$, the morphism $\phi_G$ to $\phi_G^\tr$ and the morphism $\psi_G$ to $\psi_G^\tr$.
\end{proposition}

 \begin{proof}
The construction of $\phi_G$ and $\psi_G$ from $\alpha_G$ together with 
the fact that $L\gamma^*$ is monoidal (Lemma~\ref{lemma_transfers}~(\ref{transfers_tensor})) shows that it is enough the statement for $\alpha_G$. As in the  proof of Proposition \ref{prop: with without}, it is then enough to treat the effective statement. 

We have a natural commutative diagram of sheaves without
transfers
\[\xymatrix{
\Q(G)\ar[r]^{\alpha_{G}} \ar[d] & \ul{G}_\Q \ar[d]^\sim\\
\gamma_*\Q(G)^\tr\ar[r]^{\gamma_*\alpha_G^\tr} &\gamma_*(\ul{G}_\Q)^\tr
}\]
where the left vertical map is the unit map $\epsilon_{\Q(G)}$. This diagram induces a commutative diagram in $\DAeff(S)$ whose left vertical map is the derived unit map $\epsilon_{M_S^\eff(G)}$. We apply the functor $L\gamma^*$, and stack the resulting diagram with a diagram coming from the naturality of the derived counit maps:
\[\xymatrix@C=1.8cm{
M^\eff_S(G)^\tr\ar[r]^{L\gamma^*\alpha_G} \ar[d]_{L\gamma^*\epsilon_{M_S^\eff(G)}} & L\gamma^*M_1^\eff(G) \ar[d]^\sim \\
L\gamma^*\gamma_*M^\eff_S(G)^\tr\ar[r]^{L\gamma^*\gamma_*\alpha_{G}^\tr} \ar[d]_{\eta_{M^\eff_S(G)^\tr}} & L\gamma^*\gamma_*M^\eff_1(G)^\tr \ar[d]_{\eta_{M_1^\eff(G/S)^\tr}}^\sim \\
M^\eff_S(G)^\tr\ar[r]^{\alpha_G^\tr} & M^\eff_1(G)^\tr.
}
\]
The left vertical composition is the identity by the triangular identity for the object $M_S^\eff(G)$ and the right vertical composition is the isomorphism of Proposition \ref{prop: with without}. This finishes the proof. 
\end{proof}

\section{Statement of the main theorem}\label{sect3}

\begin{definition} 

\label{def:main_def}
Let $\Mof^{(\eff)}(G/S)$, $\alpha^{(\eff)}_{G/S}$ be as in Definition~\ref{def:main_defM1}. For any integer $n \geq 0$ write $\Delta^n_G$ for the $n$-fold diagonal immersion.
\begin{enumerate}
 \item We define $\vp^{(\eff)}_{n,G}$ to be the morphism
\[ \vp^{(\eff)}_{n,G}:  M_S^{(\eff)}(G) \sxra{M^{(\eff)}_S(\Delta_{G}^{n})} M^{(\eff)}_S(G)^{\otimes n} \sxra{\alpha_{G/S}^{(\eff)\otimes n}} \Mof^{(\eff)}(G/S)^{\otimes n}.\]
As $\Delta^n_G$ is invariant under permutations, this factors uniquely 
\[\xymatrix{
M^{(\eff)}(G)\ar[rr]^{\alpha_{G/S}^{(\eff)\otimes n}\circ M^{(\eff)}_S(\Delta_G^n)}\ar[rd]_{\vp^{(\eff)}_{n,G}}&&   \Mof^{(\eff)}(G)^{\otimes n}\\
   &  \Sym^n(\Mof^{(\eff)}(G))\ar[ur]
}\]

\item Let $\kd(G/S)$ be the integer in Definition \ref{globalKimura}.
Define the morphism  
\[ \vp^{(\eff)}_{G/S}= \sum_{n = 0}^{\kd(G/S)} \vp^{(\eff)}_{n,G}: M_S^{(\eff)}(G) \longrightarrow \bigoplus_{n = 0}^{\kd(G/S)}\Sym^n(\Mof^{(\eff)}(G/S)).\]

 \item Let $p_G$ be as in Lemma \ref{lemma:connected_comp}. Define the morphism
\[\psi^{(\eff)}_{G/S} : M^{(\eff)}_S(G) \longrightarrow \left(\bigoplus_{n=0}^{\kd(G/S)} \Sym^n\Mof^{(\eff)}(G/S)\right) \otimes M^{(\eff)}_S(\pi_0(G/S))\]
to be $\psi^{(\eff)}_{G/S}= (\vp^{(\eff)}_{G/S} \otimes M^{(\eff)}_S(p_G)) \circ M^{(\eff)}_S(\Delta^2_G)$.
\end{enumerate}
\end{definition}

\begin{remark}
The K\"unneth formula holds by Proposition \ref{lemma_motives}(\ref{tensor_kunn}). This is used in the definition above to get a morphism
\[M^{(\eff)}_S(\Delta^n_G):M^{(\eff)}_S(G)\rightarrow M^{(\eff)}_S(G\times_S\cdots \times_S G)= M^{(\eff)}_S(G)^{\otimes n} \ , \] as well as to define $\psi^{(\eff)}_{G/S}$.

The morphisms $\vp^{(\eff)}_{G}$ and $\psi^{(\eff)}_{G/S}$ are the unique extensions of $\alpha_G$ compatible with the natural comultiplication on both sides; see \cite[Section 3.2]{AEH} for a more detailed discussion.
\end{remark}

The main result of the paper is the following theorem.

\begin{theorem}\label{MainThm}
Let $S$ be a noetherian finite dimensional scheme and $G/S$ a smooth commutative group scheme of finite type over $S$.

For $m\in \Z$, let $[m]:G\to G$ be the morphism of multiplication by $m$. Let $\Mof(G/S)\in \DA(S)$ be the motive from Definition \ref{def:main_defM1}, $\psi_{G/S}$ be the map as in Definition \ref{def:main_def}, $o(\pi_0(G/S))$ and $\kd(G/S)$ be the integers in Definitions \ref{orderpi} and \ref{globalKimura}. Then the following statements hold.
\begin{enumerate}
\item The relative motive $\Mof(G/S)$ is odd of Kimura dimension $\kd(G/S)$.
\item \label{main_iso} The map 
\[\psi_{G/S} : M_S(G) \longrightarrow \left(\bigoplus_{n= 0}^{\kd(G/S)} \Sym^n\Mof(G/S)\right) \otimes M(\pi_0(G/S))\]
is an isomorphism of motives. It is natural in $G \in \CGS_S$, it commutes with base change (in $S$) and it respects the natural structures of Hopf algebras.
\item The motives $\Mof(G)$ and $M(\pi_0(G/S))$ are geometric motives. 
\item \label{decmultn} The direct factor \[\mathfrak{h}_n(G/S)=\psi_{G/S}^{-1}\left(\Sym^n\Mof(G/S) \otimes M_S(\pi_0(G/S))\right)\] of $M_S(G)$  is intrinsically characterized as follows: for $m\in\Z$ that is equal to $1$ modulo $o(\pi_0(G/S))$, the map $M_S([m])$ operates on $\mathfrak{h}_n(G/S)$ as $m^n \id$. 
\end{enumerate}
\end{theorem}

\begin{remark} $M_S(G)$ and $M(\pi_0(G/S))$ carry a Hopf algebra structure
because $G/S$ and $\pi_0(G/S)$ are group objects.
The Hopf algebra structure on $\bigoplus_{n\geq 0}\Sym^n\Mof(G/S)$ is
the one of the {\em symmetric coalgebra}; see \cite[Appendix B]{AEH}.
It is isomorphic but not identical to the symmetric Hopf algebra.
\end{remark}

\begin{remark}\label{remarkeffective}
We expect the morphism $\psi^{\eff}_{G/S}$ to be already  an isomorphism. Most steps of the proof take place in $\DAeff(S)$. When $S$ is of characteristic 0, the result of \cite[Annexe B.]{Ayoub_Galois_1} can be used to show that $\psi^{\eff}_{G/S}$ is an isomorphism; since this requires some extensions of results from stable to effective motives, which are easy but not in the literature, we do not write the proof here.  An effective proof in general would require a comparison between
effective \'etale motives with and without transfers over an arbitrary field, see
the discussion in Appendix \ref{app_transfers}. An alternative approach
would be to try to redo \cite{AEH} in the category of motives without transfers.
There the missing ingredient is a special case of the comparison, namely a transfer-free computation of the effective motivic cohomology of curves.
\end{remark}

\begin{remark}
The theorem above holds in the more general context of commutative algebraic group spaces; see Section \ref{sect:alggroups}.
\end{remark}

As a consequence, we also get a version of the theorem for motives with transfers.

\begin{theorem}\label{MainThmwithtransfers}
Let $S$ be an excellent (noetherian and finite dimensional) scheme and $G/S$ a smooth commutative group scheme of finite type over $S$. Let $\psi_{G/S}$ be the map as in Definition \ref{def:main_def} and $L\gamma^*$ the functor defined in \S \ref{subsection with without}. Let $M_S(G)^{\tr}$ and $M_S(\pi_0(G/S))^{\tr}$ be the relative motives with transfers of $G$ and $\pi_0(G/S)$  and $\Mof(G/S)^{\tr}\in \DM(S)$ as in Definition \ref{M1 with transfers}.
Then 
\[L\gamma^*\psi_{G/S} : M_S(G)^{\tr} \isocan \left(\bigoplus_{n = 0}^{\kd(G/S)} \Sym^n\Mof(G/S)^{\tr}\right) \otimes M_S(\pi_0(G/S))^{\tr} \]
is an isomorphism and the analogous of Theorem \ref{MainThm} holds.
\end{theorem}
\begin{proof}
First apply $L\gamma^*$ to the isomorphism in Theorem \ref{MainThm}. Then note that by Proposition \ref{prop: with without} we have $L\gamma^*\Mof(G/S) \cong \Mof(G/S)^{\tr}$ and  by Lemma \ref{lemma_transfers}~(\ref{transfer on motives}) we have $L\gamma^*M_S(G)\cong M_S(G)^{\tr}$ and $L\gamma^*M_S(\pi_0(G/S)) \cong M_S(\pi_0(G/S))^{\tr}$.
\end{proof}

\begin{remark}\label{other categories}
\begin{enumerate}

\item Using Theorem \ref{SVqfh} one can give an alternative description of $L\gamma^*\psi_{G/S}$ using qfh sheafification. Recall that the morphism $\psi_{G/S}$ is formally constructed from a morphism of sheaves $a_{G/S}\otimes \Q$ (Definition \ref{def:sheaf_G}). If one replaces in this formal construction $a_{G/S} \otimes \Q$ by its qfh sheafification and $\DA(S)$ by $\DM(S)$ one ends with $L\gamma^*\psi_{G/S}$.

\item By the work of Cisinski and D\'eglise, the different categories of motives are related by functors. By replacing $L\gamma^*$ by those functors one obtains analogous results in the other categories of motives. 

More precisely, for the category  $\DM_h(S,\Q)$ of h-motives use \cite[Theorem 16.1.2]{CD} (one needs to suppose $S$ excellent noetherian and finite dimensional) and for the category  $\DM_{\Bei}(S)$ of Beilinson motives use \cite[Theorem 16.2.18]{CD} ($S$ noetherian and finite dimensional). 
\end{enumerate}
\end{remark}

\begin{theorem}\label{corchow}
Suppose that $S$ is regular. The $i$-th Chow group of $G$ with rational coefficients decomposes as
\[\CH^i(G)_{\Q}=\bigoplus_{j=0}^{\kd(G/S)}\CH^i_j(G),\]
where \[\CH^i_j(G)=\{Z \in\CH^i(G)_{\Q} \, | \, [m]^*Z=m^j Z \, , \forall m \equiv 1 \, (\modulo\ o(\pi_0(G/S))) \} \ .\]
\end{theorem}

\begin{proof}
Since $S$ is regular noetherian, by \cite[Corollary 14.2.14]{CD} (and \cite[20.1]{Ful} for the comparison between Chow groups and K-theory), there is a canonical isomorphism 
\[\Hom_{\DM_{\Bei}(S)}(M_S(G),\Q(i)[2i])\cong \CH^i(G)_{\Q} \ .\]
Moreover, there is a canonical equivalence between $\DM_{\Bei}(S)$ and $\DA(S)$ \cite[Theorem 16.2.18]{CD}. Hence the decomposition $M_S(G)=\bigoplus_{i=0}^{\kd(G/S)}\mathfrak{h}_n(G/S)$ of Theorem \ref{MainThm} implies the one in the statement.
\end{proof}

\begin{remark}
\begin{enumerate}
\item In the same way one has a decomposition of the higher Chow groups or of the Suslin homology.
\item Some of the eigenspaces $\CH^i_j(G)$ above should be zero. For example, if $S$ is the spectrum of a field and $G$ is an abelian variety of dimension $g$, then Beauville \cite{Beau} proved that $\CH^i_j(G)$ is zero if $j< i$ or if $j> g+i$, and that, in general, $\CH^i_j(G)$ is non-zero for $i\leq j \leq 2i$. Conjecturally $\CH^i_j(G)$ should be zero if $j> 2i$, this is part of the (still open) Bloch-Beilinson-Murre conjecture.

In our general setting one can hope to show similar vanishings of some of the eigenspaces $\CH^i_j(G)$ following the methods of Sugiyama \cite{Sug} for semiabelian varieties.
\end{enumerate}
\end{remark}

\section{Proof of the main theorem}\label{sect:proof}
\subsection{The case over a perfect field}\label{subsect:field}
In this section let $S=\Spec(k)$ be the spectrum of a perfect field. Recall that in this case $\pi_0(G/k)$ is a group scheme by Lemma \ref{lemma:connected_comp} so we do not have to consider motives of algebraic spaces in this section. We show Theorem \ref{MainThm} for all commutative group schemes of finite type $G$ over $k$. This is essentially the main result of \cite{AEH}, with exception that in loc. cit. the authors work in the effective category of motives with transfers and in this paper we primarily work in the stable category of motives without transfers. The point is then to compare the two approaches. 

\begin{proposition}\label{cor:field_case} Let $G$ be a smooth commutative group scheme over $k$. 
Then $\Mof(G)$ is odd of Kimura dimension $\kd(G/k)$ and the morphism $\psi_{G/k}$ is an isomorphism.
\end{proposition}
\begin{proof}

By Theorem \ref{comp_transfers} and Proposition~\ref{prop: with without} 
it suffices to show that
\[ L\gamma^*\Sym^n(\Mof(G))=\Sym^*(L\gamma^n\Mof(G))=\Sym^n(\Mof(G)^\tr)\]
vanishes for $n>\kd(G)$ and that
\[ L\gamma^*\psi_{G/k}=\psi_{G/k}^\tr\]
 is an isomorphism. 

As $L\Sigma^\infty_{\tr}$ is monoidal (see \ref{lemma_transfers}(\ref{transfers_kunn})) and commutes with direct sums, we have an isomorphism 
\[ \bigoplus_{n=0}^k\Sym^n M_1(G)^\tr\simeq L\Sigma^\infty_\tr\left(\bigoplus_{n=0}^k\Sym^n M_1(G)^{\eff,\tr}\right) .\]
 Modulo this isomorphism, $\psi_G^\tr$ agrees with the morphism deduced by applying $L\Sigma^\infty_\tr$ to
 \[\psi^{\eff,\tr}_{G/k} : M^{\eff}_k(G)^\tr \longrightarrow \left(\bigoplus_{n\geq 0} \Sym^n\Mof^{\eff}(G/k)^\tr\right) \otimes M_k^{\eff}(\pi_0(G/k))^\tr\ .\]
 

 On the other hand, by the uniqueness part in Theorem \ref{G with transfers}, the morphism $\psi^{\eff,\tr}_{G/k}$ is exactly the morphism considered in \cite[\S 7.4]{AEH} (although with different notation). We are now able to apply \cite[Theorem 7.4.6]{AEH}, which concludes the proof.
\end{proof}

\subsection{The general case}

We return to an arbitrary base scheme $S$.

\begin{proof}[Proof of Theorem \ref{MainThm}.]
\begin{enumerate}
\item We consider the Kimura dimension $\kd(G/S)$ of $\Mof(G/S)$. We claim that
\[ \Sym^n(\Mof(G/S))=0\]
for $n>\kd(G/S)$. By Lemma \ref{vanishing}, we can test this after pullback
to all geometric points $i_{\bar{s}}:\bar{s}\to S$. The functor $i_{\bar{s}}^*$ commutes
with tensor product and hence with $\Sym^n$. By Proposition \ref{pullback_group} we have $i_{\bar{s}}^*\Mof(G/S)=\Mof(G_{\bar{s}})$. By definition 
$\kd(G/S)\geq \kd(G_{\bar{s}})$. Hence the vanishing holds by Proposition~\ref{cor:field_case}.

 \item By Lemma \ref{lemma:connected_reduction}(\ref{M1linear}) the morphism $\alpha_{G/S}$ is natural in $G \in \CGS_S$. This implies naturality for the morphism $\psi_{G/S}$. Naturality implies that the morphism $\psi_{G/S}$ is a morphism of Hopf algebras by the same argument as in the absolute case; see \cite[Proposition 3.2.9, Theorem 7.4.6]{AEH}.

By Proposition \ref{pullback_group} the morphism $\alpha_{G/S}$ commutes with base change. As the pullback on motives is a monoidal functor (Lemma \ref{lemma_motives} (\ref{pullback_tensor})) the morphism $\vp_{G/S}$ also commutes with base change. Since the formation of $\pi_0(G/S)$ commutes with base change, so does $\psi_{G/S}$.

 We turn to the claim that $\psi_{G/S}$ is an isomorphism.
By Lemma~\ref{vanishing}, it suffices to check the assertion after
pullback via $i_{\bar{s}}:\bar{s}\to S$ for all geometric points $\bar{s}$ of $S$. On the other hand
$i_{\bar{s}}^*\psi_{G/S}=\psi_{G_{\bar{s}}/\bar{s}}$, as the map $\psi_{G/S}$ commutes with base change. This is an isomorphism
by Proposition \ref{cor:field_case}.

\item If $\pi_0(G)=0$, then $\Mof(G)$ is geometric because it is a direct summand 
of a geometric object by part \ref{main_iso}. This also implies the general case because $\Mof(G)=\Mof(G^0)$. Finally, $M(\pi_0(G))$ is geometric because it is a direct factor of $M_S(G)$ by part \ref{main_iso}.

\item By Lemma \ref{lemma:connected_reduction}(\ref{M1linear}), we have $\Mof([m])=m \cdot \id$ for all integers $m$, so that $\Sym^n\Mof([m])=m^n\id$ for all $n\in \N$. To conclude, note that $m$ is equal to $1$ modulo $o(\pi_0(G/S))$, the multiplication by $m$ is the identity on the space $\pi_0(G/S)$. 

\end{enumerate}
\end{proof}

\section{Realizations} \label{section:real}
We want to study the image of our decomposition under realization
functors. We use the existence of realizations
functors compatible with the six functor formalism.
We have decided against axiomatizing the statement but
rather treat the explicit cases of Betti and $\ell$-adic realization.

\subsection{Betti realization}
In this section, we assume all schemes are of finite type over $\C$. For such a scheme $S$ we denote by
$S^\an$ the associated complex analytic space, equipped with its natural topology. Let $D(S^{\an},\Q)$ be the derived category of sheaves of $\Q$-vector spaces on $S^\an$ and $D^b_c(S^\an,\Q)$ be the subcategory of bounded complexes with constructible cohomology.
By \cite{Ayoubbetti} (see also \cite[17.1.7.6]{CD} for an elaboration in terms of ring spectra)
there is a system of covariant functors
\[ \Betti: \DA(S)\to D(S^\an,\Q)\]
compatible with the six functor formalism on both sides (note that some commutativity morphisms are shown to be isomorphisms only when applied to constructible motives). It maps
the Tate object $\Q(j)$ to $\Q$. Moreover, as the relative Betti homology of a smooth $S$-scheme lies in $D^b_c(S^\an,\Q)$, constructible motives are sent to $D^b_c(S,\Q)$.

\begin{proposition}\label{Betti}Let $\pi:G\to S$ be a smooth commutative group scheme of relative dimension $d$ with
connected fibres. Then:
\begin{enumerate}
\item $\Betti (M_S(G))=\pi^{\an}_!(\pi^{\an})^!\Q_S=\pi^{\an}_!\Q_G[2d]$
\item $\Betti (\Mof(G)[-1])= R^{2d-1}\pi^{\an}_!\Q_G$ with fibre in $s\in S^\an$
given by $H_1(G_s,\Q)$. We set $\Hh_1(G/S,\Q)=\Betti (\Mof(G)[-1])$.
\item \label{dec_Betti} $\pi^{\an}_!\Q_G(d)[2d]=\bigoplus_{i=0}^{\kd(G/S)}(\bigwedge^i\Hh_1(G/S))[i]$.
\end{enumerate}
\end{proposition}
\begin{proof}
First notice that $\pi$ is separated by \cite[Expos\'e VI$_B$ Corollaire 5.5]{SGA3} so that the exceptional operations $\pi_!$ and $\pi^!$ are well-defined. We have
\[ M_S(G)=\pi_\#\Q_G=\pi_!\pi^!\Q_S\]
because $\pi$ is smooth. The first assertion then follows from the same statement in $\DA(S)$ and the compatibility theorem \cite[Theoreme 3.19]{Ayoubbetti}.

By our main Theorem \ref{MainThm}, $\Betti\Mof(G)$ is a direct factor of $\Betti(M_S(G))\simeq \pi_!\pi^!\Q_S$. Its fibre over $s\in S^\an$ is given by
$\Betti(\Mof(G_s))$.  This was computed in \cite[Proposition 7.2.2]{AEH} via an Hopf algebra argument which also applies in our setting with the difference that in loc. cit. the realization functor was assumed to be {\em contravariant}, i.e., $M(G_s)$ was mapped to $\pi_*\pi^*\Q$. In that language, the
realization of $\Mof(G_s)$ was concentrated in degree $1$ and equal to
$H^1(G_s,\Q)$ as a direct factor of $H^*(G_s,\Q)$. Hence in the present setting, the realization is concentrated
in degree $-1$ and equal to $H_1(G_s,\Q)$. 

This shows that indeed $\Betti(\Mof(G)[-1])\simeq R^{2d-1}\pi^\an_!\Q_G$. 

The last statement follows 
by functoriality of $\Betti$.
\end{proof}
\begin{remark}
\begin{enumerate}
\item $\Hh_1(G/S)$ has an explicit description as 
\[ \Hh_1(G/S)=\Ker\left(\ul{\Lie(G^\an)}\xrightarrow{\exp}\ul{G^\an}\right)\tensor\Q\ .\]
A similar description holds integrally. This can be shown directly in the setting of constructible sheaves or
from an explicit computation of $\Betti(\Mof(G))$, see the upcoming
thesis \cite{thesis-simon} of the third author.
\item The assumption on the connectness of fibres is not needed, but simplifies
notation for the result. 
\end{enumerate}
\end{remark}

\begin{remark}
\label{remark:real_common}
\begin{enumerate}
\item The statement in Proposition \ref{Betti} (\ref{dec_Betti}) means that our canonical motivic decomposition is a decomposition
into relative K\"unneth components with respect to the conjectural standard
motivic $t$-structure.
\item If $G$ has constant abelian rank and torus rank over a regular scheme $S$, then
$\Hh_1(G/S,\Q)$ is a local system and thus (up to a shift) a
perverse sheaf.
We do not know if this is true in general and expect it to be false; however see the following example.
\end{enumerate}

\end{remark}
\begin{example}\label{example:real}
Let $S$ be regular of dimension $1$ and $G/S$ be a degenerating family
of elliptic curves with multiplicative or additive reduction. In both
cases $\Betti(\Mof(G/S))$ is a perverse sheaf.
\end{example}

\subsection{$\ell$-adic realization}
In this section, we fix a prime $\ell$ and assume that all schemes are $\Z[\frac{1}{\ell}]$-schemes.   
For such a scheme $S$, let $D_c(S,\Ql)$ be subcategory of complexes with constructible cohomology in the
derived category of $\Ql$-sheaves $S$ in the sense of Ekedahl \cite{Eke}.
By \cite[Section 9]{Ayoubet} there are covariant functors 
\[ \ladic: \DA_c(S,\Q)\to D_c(S,\Ql)\]
compatible with the six functor formalism on both sides. It maps
the Tate object $\Q(j)$ to $\Ql(j)$.

For a smooth group scheme $G/S$ with connected fibres, we can deduce, by the same argument as in Proposition \ref{Betti},
an analoguous decomposition for $\ladic(M_S(G))$. The corresponding object $\Hh_1(G/S, \Ql)$ has fibre over
$s\in S$ given by the rational Tate module $V_\ell(G_s)$. The analogues of
Remark \ref{remark:real_common} also hold in this setting.

\begin{remark}
The $\ell$-adic sheaf $\Hh_1(G/S,\Ql)$ is given by the system
of constructible torsion sheaves $G[\ell^n]$. This also holds integrally. Again this can be seen either directly in $\ell$-adic cohomology or via a computation of the realization of $\Mof(G)$;
see \cite{thesis-simon}.
\end{remark}

\subsection{Hodge realization}

Let $S$ be separated of finite type over $\C$. We expect the existence of
Hodge realization functors
\[ \Hodge:\DA_c(S)\to D^b(\MHM(S))\]
(where $\MHM(S)$ is Saito's category of mixed Hodge modules over $S$)
compatible with the six functor formalism and such the forgetful functor
to the underlying derived category of sheaves is equal to $\Betti$.
This would yield a refinement of Proposition \ref{Betti}. Note that
$\Hodge(\Mof(G/S))$ is a mixed Hodge module (up to shift) if and only if
$\Betti(\Mof(G/S))$ is a perverse sheaf (up to shift).

Ivorra has constructed a functor $\Hodge$ for $S/\C$ smooth quasi-projective \cite{Ivorra_Hodge}, but the compatibility
with the six functors (in particular the tensor functor!) is still open. Hence we get a partial result. There is also upcoming work of Drew \cite{drew-thesis,drew-oberwolfach}
defining a Hodge realization $\Hodge'$
with values in some triangulated category $D_\mathrm{Hdg}(S)$ compatible with the six functor formalism. It is known that $D_\mathrm{Hdg}(\C)\simeq D^b(\mathrm{MHS})$ (so that $\Hodge'$ puts in family the Hodge structures on the (co)homology of fibres) but the comparison with mixed Hodge modules over a more general $S$ is still not understood.

\section{Algebraic group spaces}\label{sect:alggroups}

We extend our main result from commutative group schemes to
commutative algebraic group spaces. First we discuss generalities and some examples.

Recall (see e.g. \cite[Chapter 1]{LMB}) that an algebraic space $Y$ is an 
\'etale sheaf of sets over the site $\Sch/S$ with the \'etale topology induced by a scheme $X$ and an \'etale equivalence relation $R$.
This means that $Y$ is the cokernel of the diagram of \'etale sheaves of sets
\[ R\rightrightarrows  X\ .\]
It is {\em smooth} (resp. {\em \'etale}) if $X$ is smooth (resp. \'etale) over $S$. An algebraic group space is defined to be a group object in the category of algebraic spaces (which does not mean that we can find a presentation as above with $X$ a group and $R$ a subgroup).

Algebraic group spaces are closer to schemes than general algebraic spaces.

\begin{prop}{\cite[Lemma 4.2]{Artin_FM1}}\label{prop:alg_space_field} If $S$ is the spectrum of a field and $G/S$ an algebraic group space of finite type over $S$, then $G$ is a scheme.
\end{prop}

By a spreading-out argument, one deduces the following seemingly stronger result.

\begin{corollary}\label{coro:stratification}
Let $S$ be a noetherian scheme and $G/S$ an algebraic group space of finite type. Then exists a stratification of $S$ such that $G$ restricted to any stratum is a scheme.
\end{corollary}

Commutative algebraic group spaces that are not group schemes appear naturally in algebraic geometry. We have already seen the example of the \'etale algebraic space $\pi_0(G/S)$ in Lemma \ref{lemma:connected_comp}, which was an instance of a very general result of Artin of representability of quotients \cite[Corollaire 10.4]{LMB}. In the same vein, we have the following useful fact.
\begin{lemma}{\cite[Proposition 4.6]{FK}}\label{lemma:etale_alg_spaces}
Let $S$ be a scheme. The constructible \'etale sheaves of sets (resp. of abelian groups) are exactly the \'etale algebraic spaces (resp. \'etale algebraic group spaces) of finite type over $S$.
\end{lemma}

A fundamental source of examples are Picard functors of proper flat cohomologically flat morphisms \cite[8.3]{Neron}; they are not finite type, smooth or separated in general, but see \cite[8.4]{Neron}.

Finally, the Weil restriction along a finite flat morphism sends (smooth) algebraic group spaces to (smooth) algebraic group spaces (see \cite[7.6]{Neron} and \cite[Theorem 1.5]{Olsson_res}).

For the rest of this section, we adopt the following shorthand
\begin{definition}\label{defn:groupspace}A {\em group space} is a smooth commutative algebraic group space of finite type. 
\end{definition}

\begin{lemma}\label{lemma:alg_comp} Let $G$ be a group space. Then there is a short exact sequence
of group spaces
\[ 0\to G^0\to G\to \pi_0(G/S)\to 0\]
such that all fibres of $G^0$ are connected and $\pi_0(G/S)$ is 
an \'etale group space. The formation of this sequence commutes with base change.
\end{lemma}

\begin{proof}
The union $G^0$ of neutral connected components of fibres is an open subset of the set of points of the algebraic space $G$ by applying \cite[Proposition 2.2.1]{Romagny_connected} to the zero section of $G$. Hence $G^0$ corresponds to an open group subspace of $G$. One then defines $\pi_0(G/S)$ as the quotient \'etale sheaf $G/G^0$ and proceeds as in the proof of Lemma \ref{lemma:connected_comp}.  (The paper \cite{Romagny_connected} describes an \'etale group space $\pi_0(G/S)$ via its functor of points (see \cite[Definition 2.1.1.(1)]{Romagny_connected} and \cite[2.5.2.(i)]{Romagny_connected}). One can show that it coincides with the one in the Lemma above using \cite[2.5.2.(ii)]{Romagny_connected}.)
\end{proof}

We define the Kimura dimension $\kd(G)$ of $G$ and the order of $\pi_0(G/S)$ as in Section \ref{sect1}. The Kimura dimension is well-defined: assume $G$ is represented by a smooth scheme
$X$ with relations in $R$. Then the fibre dimension of $G$ is bounded
by the fibre dimension of $X$ and hence $\kd(G_s)$ is bounded by $2\dim_S X$.
The order is well-defined by Lemma \ref{lemma:alg_comp} and the fact that the sheaf $\pi_0(G/S)$ is constructible by Lemma \ref{lemma:etale_alg_spaces}.

\begin{remark}
We do not know if Lemma \ref{lemma:kd_semi_cont} holds in this more general setting. If $G/S$ is separated, the proof given there works because the group spaces $G[l^n]$ are then schemes by \cite[II, Theorem 6.15]{Knutson}. There are non-separated group spaces over discrete valuation rings, but in standard examples the non-separatedness comes from removing connected components in the special fibre, which does not affect $\kd(G/S)$. 
\end{remark}

Let
$M^\eff_S(G)$ (resp. $M_S(G)$) be the effective motive (resp. motive) of $G$ as in Appendix \ref{app_relative}.
\begin{lemma}\label{alggroup_is_geometric}
$M_S(G)$ is geometric.
\end{lemma}

\begin{proof}
By Corollary \ref{coro:stratification} we have a stratification such that for any stratum $i:T\rightarrow S$, the restriction $G_T$ is a scheme. We have $i^*M_S(G)=M_T(G_T)$ by Proposition \ref{pullback_group}. The result then follows by induction on strata and localisation.
\end{proof}

In analogy with the scheme case, we define $\Mof^\eff(G/S)$ to be the image
of $G_{|\Sm/S}\tensor\Q$ in $\DM^\eff(S)$ and $\Mof(G/S)$ to be the object $L\Sigma^\infty \Mof^\eff(G/S)$ in $\DM(S)$. The definitions
of $\alpha_{G/S}$, $\phi_{G/S}$ and $\psi_{G/S}$ in Definition \ref{def:main_defM1} and Definition \ref{def:main_def}
work without changes.

\begin{theorem}\label{MainThm_alg_space}
Let $G$ be a group space in the sense of Definition \ref{defn:groupspace}. Then all assertions of Theorem \ref{MainThm} hold true.
\end{theorem}

\begin{proof}
Let $f:T\to S$ be a morphism of schemes.
By Proposition \ref{pullback_algspace}, the pullback
of $M_S(Y)$ is given by $M_T(Y_T)$ when $Y$ is an algebraic space.
Hence the same argument as in the proof of Proposition \ref{pullback_group} shows 
that $f^*\Mof(Y/S)=\Mof(Y_T/T)$. Again we show that $f^*\Mof(G/S)=\Mof(G_T/T)$ using the resolution
obtained by Theorem \ref{theorem:resolution}. It is not clear if its terms are cofibrant. However, by the above remark we understand their pullback, which is enough.

Now all steps in the proof work in the same way as in the scheme case, i.e, by
reducing to the case of $S=\Spec\ k$ with $k$ an algebraically closed
field. 
By Proposition \ref{prop:alg_space_field}, we are then back in the case of group schemes settled by Proposition \ref{cor:field_case} (in particular we do not need to consider transfers on algebraic group spaces).
\end{proof}

\begin{remark}Applying the realization functors $\Betti$ or $\ladic$, we
also obtain a computation of the  K\"unneth components. In particular, $\Betti\Mof(G)$ and $\ladic\Mof(G)$
are concentrated in degree $-1$. Conceptually this means that $\Mof(G)$ should still be a homological $1$-motive when $G$ is a group space. 
Note, however, that the results are a bit weaker than in Section \ref{section:real} because there are not as yet categories of triangulated motives over algebraic spaces satisfying the six functor formalism.
\end{remark}

\begin{appendix}

\section{Generalities on relative motives}\label{app_relative}
Let $S$ be noetherian and finite dimensional. There are various approaches to the theory of mixed motivic sheaves over $S$, even when working with rational coefficients. We refer to the end of the introduction for a brief discussion. We use the category of \'etale motives without transfers introduced originally by Morel and studied by Ayoub with the notation $\DA^\et(S,\Q)$. In the work of Cisinski and D\'eglise (see \cite[16.2.17]{CD}) $\DA^\et(S,\Q)$
is the category $\D_{\mathbb{A}^1,\et} (\Sm/S,\Q)$. By \cite[Theorem 16.2.18]{CD}
it is equivalent to their category of {\em Beilinson motives}. Since we systematically use the \'etale topology and rational coefficients, we simply write $\DA(S)$.

\begin{definition}Let $X$ be a smooth $S$-scheme. We write $\Q(X)\in\Sh_\et(\Sm)$ for the sheafification of the presheaf
\[\Q\Mor_S(-, X) \ ,\]
which associate to each smooth $S$-scheme $Y$ the $\Q$-vector space with basis the set of $S$-morphisms from $Y$ to $X$. 
\end{definition}

Write $T$ for the sheaf cokernel of the unit section $\Q(S)\rightarrow \Q(\G_{m,S})$; we view $T$ as a convenient model for the Tate motive. The category $\DA(S)$ is defined to be the homotopy category of a certain stable $\A^1$-local model category of symmetric $T$-spectra of complexes in the abelian category $\Sh_\et(\Sm)$ of \'etale sheaves in $\Q$-vector spaces on $\Sm/S$. For simplicity, we refer to an object in this model category as a {\em motivic spectrum}. The analogous construction without spectra leads to the category $\DAeff(S)$ of {\em effective triangulated motives}, which can be viewed as an $\mathbb{A}^1$-localization of the derived category of $\Sh_\et(\Sm)$. For simplicity, we refer to an object in this category as a {\em motivic complex}. 

\

Recall that to compute derived functors in this setting, one chooses a specific model category structure on (spectra of) complexes of sheaves presenting the categories of motives such that the functors of interest are Quillen functors. For this paper, we use indifferently the projective model structures (adapted from presheaves to sheaves) of \cite[Section 3]{Ayoubet} or the descent model structures of \cite[5.1.11]{CD}.

There is a sequence of functors (see \cite[5.3.23.2]{CD} for more details on the last one, the {\em infinite $T$-suspension functor})
\begin{equation}\label{functors}\Sh_\et(S)\to D(\Sh_\et(\Sm))\to\DAeff(S)\stackrel{L\Sigma^\infty}\to\DA(S).\end{equation}

\begin{definition} Let $X$ be a smooth $S$-scheme.
\begin{enumerate}
\item The {\em (relative, homological) effective motive} $M_S^\eff(X)$ is defined as the image of $\Q(X)$ in $\DAeff(S)$.
\item The motive $M_S(X)$ is defined as $L\Sigma^\infty M_S^\eff(X)$ in $\DA(S)$.
\end{enumerate}
Motivic complexes (respectively, spectra) of the form $\Q(X)$ (respectively, $\Sigma^\infty\Q(X)$)  are called {\em representable}.

The category $\DAgm(S)$ of \textit{constructible} or \textit{geometric}  motives over $S$ is defined as the thick triangulated subcategory of $\DA(S)$ generated by the
motives of the form $M_S(X)(i)$ for any smooth $S$-scheme $X$ and any integer $i \in \Z$.
\end{definition}
\begin{remark}
\label{rem:repr_cofibrant}
The projective model structures that we are using have relatively few cofibrant objects - as they should, since we want to use them to compute left derived functors - but crucially representable objects are cofibrant (by definition for the descent model structure of \cite[5.1.11]{CD}, and as they have the left lifting property with respect to surjective quasi-isomorphisms for the projective one).
\end{remark}

\begin{remark}
A motive is constructible if and only if it is a \textit{compact} object of the triangulated category $\DA(S)$ (see \cite[Proposition 8.3]{Ayoubet}).
\end{remark}

By \cite{Ayoub_these_1} and \cite{Ayoub_these_2} (see also \cite{CD}), the categories $\DA(S)$ satisfy a full six functor formalism. We are only going to use explicitely the tensor product and the pullback $f^*$ for morphisms $f:T\to S$. The following proposition sums up the computational properties
we need.

\begin{lemma}\label{lemma_motives}
Let $f:S\to T$ be a morphism of noetherian schemes of finite dimension.
\begin{enumerate}
\item\label{tensor_kunn} All functors in the sequence (\ref{functors}) are monoidal.
In particular, if $X,Y$ are objects in $\Sm/S$, then 
\begin{gather*}
\Q(X)\tensor\Q(Y)\simeq \Q(X\times Y) \ , \\
M_S^{(\eff)}(X)\tensor M_S^{(\eff)}(Y)\simeq M_S^{(\eff)}(X\times Y) \ .
\end{gather*}
\item\label{pullback_representable} For $X\in\Sm/S$, we have canonical isomorphisms: 
\begin{equation*}
f^*\Q(X)\simeq\Q(X_T)\text{ and }\ f^*M_S^{(\eff)}(X)\simeq M_T^{(\eff)}(X_T) \ .
\end{equation*}
\item\label{pullback_suspension} There is a canonical $2$-isomorphism $L\Sigma^\infty f^*\simeq f^*L\Sigma^\infty.$
\item\label{pullback_tensor} The pullback functor $f^*$ is monoidal.
\end{enumerate}
\end{lemma}
\begin{proof} 
Assertion (\ref{tensor_kunn}) follows from the universal property of fibre products, the fact that the tensor product on sheaves of $\Q$-vector spaces is exact, and \cite[Corollaire 4.3.72]{Ayoub_these_2}.

Let us prove assertion (\ref{pullback_representable}). The formula for the pull-back of $\Q(X)$ follows from the universal property
of the fibre product. As the functors $f^*$ on motivic complexes and spectra are left Quillen (by combining \cite[5.1.14]{CD} and \cite[5.3.28]{CD}) and representable objects are cofibrant, the same formula also holds in the categories of effective and non-effective motives.

Assertion (\ref{pullback_suspension}) holds by \cite[Section 5.23]{CD}.

Assertion (\ref{pullback_tensor}) follows from the analoguous property for the pullback on motivic complexes (spectra), for which we refer to the proofs of \cite[Propositions 4.5.16, 4.5.19]{Ayoub_these_2}.
\end{proof}

\begin{lemma}\label{vanishing}
\begin{enumerate}
\item
Let $M\in \DA(S)$ be a motive. Then $M$ is zero if and only if the pullback $i_{\bar{s}}^*M$ to any geometric point 
\[i_{\bar{s}}: \bar{s} \longrightarrow S\]
is zero. 
\item Let $f$ be a morphism in $\DA(S)$. Then $f$ is an isomorphism if and only if the pullback $i_{\bar{s}}^*(f)$ to any geometry point $\bar{s}\to S$ is an isomorphism. 
\end{enumerate}
\end{lemma}
\begin{proof}The second statement follows from the first by the axioms
of a triangulated category.

We turn to the proof of the first. 
By \cite[Proposition 3.24]{Ayoubet} the family $i_s^*$ for all points $s\in S$ 
is conservative. Hence, we may now assume that $S=\Spec\ k$ is the spectrum of a field. Let $k^i$ be the inseparable closure of $k$. It is well-known that pull-back induces an equivalence of categories between the category of \'etale sheaves on $\Spec\ k$ and the category of
\'etale sheaves on $\Spec\ k^i$. It is much more difficult, but nonetheless true, that it induces an equivalence between the categories of motives (see \cite[Proposition 2.1.9]{CD} and \cite[Theorem 14.3.3]{CD}).

We may now assume that $k$ is perfect.
Let $\bar{k}$ be an algebraic closure of $k$. Let $P\in\DA(\Spec\ k)$
such that the pull-back $i^*$ to $\Spec\ \bar{k}$ vanishes.
As the category
$\DA(S)$ is compactly generated by \cite[Proposition 3.19]{Ayoubet}, it suffices
to show that all morphisms $f:M\to P$ with $M$ compact vanish. By assumption,
the morphism $i^*(f)$ vanishes. By \cite[Proposition 3.20]{Ayoubet} this 
implies that the pull-back of $f$ to some finite separable extension of $k$ vanishes. By \cite[Lemme 3.4]{Ayoubet} such a restriction is conservative and so $f=0$.
\end{proof}

We need to consider motives of algebraic spaces. The definition follows
the method of Choudhury \cite{utsav}, who considered the case of stacks over
a field. We only develop the minimal amount of results necessary.

Let $Y$ be an algebraic space 
presented by a scheme $X$ and an \'etale equivalence relation $R$, i.e., $Y$ is the cokernel of the diagram of sheaves of sets over the site $\Sch/S$ with the \'etale topology
\[ R\rightrightarrows X\ .\]

\begin{definition}\label{defn_algspace}Let $Y$ be a smooth algebraic space. 
Let $\Q(Y)$ be the sheaf associated to the presheaf $\Q Y_{|\Sm/S}(\cdot)$. 
We define the
relative motive $M^\eff_S(Y)\in\DAeff(S)$ as the image of the sheaf $\Q(Y)$ and $M_S(Y)\in\DA(S)$ as $L\Sigma^\infty M^\eff_S(Y)$.
\end{definition}
If $Y$ is a scheme, this agrees with the definition
given before.

\begin{proposition}\label{pullback_algspace}
Let $f:T\to S$ be a morphism and $Y/S$ a smooth algebraic space.
Then there are canonical isomorphisms
\[ f^*M_S^{(\eff)}(Y)\simeq M_T^{(\eff)}(Y_T).\]
\end{proposition}

\begin{proof}
By Lemma \ref{lemma_motives}(\ref{pullback_suspension}) it is enough to settle the effective statement.

We follow the method of \cite[Cor. 2.14]{utsav} which was considering Deligne-Mumford stacks.
Let $(X,R)$ be a presentation of $Y$ and $X_\bullet$ be the Cech nerve of the cover $X\to Y$. By general
principles, the covering map
\[ X_\bullet\to Y\]
is a weak equivalence of simplicial sheaves. Hence it induces a quasi-iso\-mor\-phism of complexes of sheaves of $\Q$-vectors spaces
\[ \Q(X_\bullet)\to \Q(Y)\ .\] 
We are going to show that $X_n$ is a smooth $S$-scheme. Assuming this, we
compute $f^*M_S(X)$ as
\[ f^*\Q(X_\bullet)=\Q( (X_\bullet)_T)\ .\]
$(X_\bullet)_T$ is the Cech nerve of the cover $X_T\to Y_T$, hence the
latter is quasi-isomorphic to $\Q(Y_T)$.

By definition, $X_n$ is
the $(n+1)$-fold fibre product of sheaves of sets
\[ X_n=X\times_Y\times X\dots\times_Y X=(X\times_Y X)\times_X\dots \times_X(X\times_X X)\ .\]
Since $X\times_YX\simeq R$ we see that $X_n$ is a smooth $X$-scheme. As $
X$ is smooth, this makes $X_n$ a smooth $S$-scheme.
\end{proof}

\section{Motives with and without transfers}\label{app_transfers}
Let $S$ be a noetherian scheme of finite dimension.

\subsection{Sheaves with and without transfers}\label{subsection sheaves with without}
\begin{definition}
\begin{enumerate}
\item
Let $\Sm^\tr$ be the category of smooth correspondences over $S$. Its objects
are the objects of $\Sm$ and morphisms are finite $\Q$-correspondences in the sense of
\cite[Definition 9.1.2]{CD}.
\item
Let $\Sh_\et(\Sm^\tr)$ be the category of additive presheaves of $\Q$-vector spaces on $\Sm^\tr$ whose 
restriction to $\Sm$ is an \'etale sheaf. We call these objects
{\em \'etale sheaves with transfers}.
\item
Let $X\in\Sm$. Then $\Q(X)^\tr$ is defined as the representable presheaf with transfers
$c_S(\cdot,X)$ (which is in fact a sheaf).
\end{enumerate}
\end{definition}
The category $\Sh_\et(\Sm^\tr)$ is monoidal. The tensor product is characterized
by the formula
\[ \Q(X)^\tr\tensor\Q(Y)^\tr=\Q(X\times Y)^\tr\]
for $X,Y\in\Sm$.

By  \cite[Corollary 10.3.11]{CD}, there is an adjoint pair of functors
\[ \gamma^*:\Sh_\et(\Sm)\rightleftarrows  \Sh_\et(\Sm^\tr):\gamma_*\ ,\]
where $\gamma_*$ is simply the restriction from sheaves with transfers to sheaves without transfer, and $\gamma^*$ verifies the following properties.
\begin{lemma}\label{lemma with without}
Let $\gamma^*$ and $\gamma_*$ the functors defined above. Then, for all $S$-smooth schemes $p: X \rightarrow S$ and all sheaves with transfers $F$, the following statements hold:
\begin{enumerate}
\item\label{transfers on X} $\gamma^*\Q(X)=  \Q(X)^\tr$,
\item\label{transfers monoidal} the functor $\gamma^*$ is monoidal.
\end{enumerate}
\end{lemma}
\begin{proof}
By \cite[Corollary 10.3.11]{CD},  the functor $\gamma_*$ induces an adjunction of abelian $\mathcal{P}$-premotivic categories \cite[Definition 1.4.6]{CD}. This implies both properties of the statement. Indeed, by definition,  $\gamma^*$ is in particular a morphism of monoidal $\mathcal{P}$-fibered categories \cite[Definition 1.2.7]{CD} (which implies \ref{transfers monoidal}) and commutes with the functors $p_{\#}$ (which implies \ref{transfers on X}).
\end{proof}

There is an alternative description of $\gamma^*$ via the category of $\qfh$-sheaves. Let $\Sch$ be the category of schemes of finite type over $S$. The inclusion of categories induces a morphism of sites $p:\Sch_\et\to\Sm_\et$ and 
a pair of adjoint functors
\[ p^*:\Sh_\et(\Sm)\rightleftarrows  \Sh_\et(\Sch):p_*\ ,\]
where $p_*$ is simply the restriction. We also write more suggestively
$F|_\Sm=p_*F$.
Recall the $\qfh$-topology on $\Sch$ from \cite[Section 4.1]{RCCS} . Roughly, it the topology generated by open covers and finite surjective morphisms. As the qfh topology refines the \'etale topology, there is a morphism of sites 
$r:\Sch_\qfh\to\Sch_\et$ and a pair of adjoint functors
\[ r^*:\Sh_\et(\Sch)\rightleftarrows\Sh_\qfh(\Sch):r_*.\]
We also write more suggestively $F_\qfh=r^*F$. In particular, $\Q(X)_\qfh$ is the sheafification of the presheaf $\Q(X)$ on $\Sch$.
As usual we denote by $\Q(X)^\tr\in\Sh_\et(\Sm^\tr)$ the representable sheaf with transfers
defined by $X$.

 The following result is due to Suslin and Voevodsky \cite[Proposition 4.2.7, Theorem 4.2.12]{RCCS}.
\begin{theorem}\label{SVqfh}
For $X\in\Sm$, the canonical map $\Q(X)\to \gamma_*\Q(X)^\tr$ factors through the qfh sheafification and induces an isomorphism:
\[\Q(X)_\qfh|_\Sm\stackrel{\sim}{\to}\gamma_*\Q(X)^\tr\ .\]
\end{theorem}
This implies that any $\qfh$-sheaf has canonical transfers; see \cite[Proposition 10.5.8]{CD}. 

A priori, the functor $\gamma^*$ is only right exact. 
(The above theorem suggests that $\gamma^*$ should be equal to $p_*r_*r^*p^*$, the source of non-exactness
being $r_*$.) 
It is not clear whether it is also left exact. However, we have the following criterion.

\begin{lemma}\label{specialgamma}
Assume $S$ normal. Let $F$ be an \'etale sheaf of $\Q$-vector spaces on $\Sch$ and 
\[ C^*\to F\]
a resolution in the same category such every $C^i$ is of the form $\Q(X_i)$ for some smooth $X_i$. Then the following holds.
\begin{enumerate}[(i)]
\item \label{exact_gamma} The complex 
\[ \gamma^*(C^*|_\Sm)\to \gamma^*(F|_\Sm)\to 0\]
of objects in $\Sh_\et(\Sm^\tr)$ is exact.
\item \label{gamma_qfh}
There is a canonical isomorphism
\[(F_\qfh)|_\Sm\stackrel{\sim}{\longrightarrow} \gamma_*\gamma^*(F|_\Sm).\]
and this isomorphism identifies the unit map $F|_{\Sm}\rightarrow \gamma_*\gamma^*(F|_\Sm)$ and the unit map $F|_\Sm\rightarrow (F_\qfh)|_\Sm$.
\end{enumerate}
\end{lemma}
\begin{proof}
The complex $C^*_{\qfh}\to F^*_{\qfh}\to 0$ is an exact complex of qfh sheaves on $\Sch$. We claim that its restriction to $\Sm$ is an exact complex of \'etale sheaves. Since $S$ (and hence every scheme in $\Sm$) is normal and we work with sheaves of $\Q$-vector spaces, this follows from a trace argument (see \cite[Theorem 3.4.1]{VHS}).

Because of the assumption on the $C^i$, Theorem \ref{SVqfh} implies that the natural map $C^*|_\Sm\ra \gamma_*\gamma^*(C^*|_\Sm)$ factors through $C^*_{\qfh}|_\Sm$ and that we get an isomorphism of complexes in $\Sh_\et(\Sm)$:
\[
\delta:C^*_{\qfh}|_\Sm\stackrel{\sim}{\longrightarrow}\gamma_*\gamma^*(C^*|_\Sm)
\]

This implies that the complex $\gamma_*\gamma^*(C^*|_\Sm)$ is exact. Since $\gamma^*$ is right exact and $\gamma_*$ is exact, we get that the complex
\[ \gamma_*\gamma^*(C^{-1}|_\Sm)\to\gamma_*\gamma^*(C^0|_\Sm)\to \gamma_*\gamma^*(F|_\Sm)\to 0\]
is also exact. Put together, this implies that $\gamma_*\gamma^*(C^*|_\Sm)\to \gamma_*\gamma^*(F|_\Sm)\to 0$ is exact. Since $\gamma_*$ is exact and faithful, this proves point (\ref{exact_gamma}). Now the isomorphism $\delta$ induces an isomorphism on $H^0$ which proves the first part of (\ref{gamma_qfh}). The second part is formal from the definition of the map.
\end{proof}

\subsection{Triangulated categories with and without transfers}\label{subsection with without}
We can now introduce the triangulated categories of motives with transfers $\DMeff(S)$ and $\DM(S)$ which are defined in a completely parallel manner to $\DAeff(S)$ and $\DA(S)$ by replacing sheaves with sheaves with transfers. They also occur as homotopy categories of model structures, for which we do not go into details and refer to \cite[Definition 11.1.1]{CD}.

We have a sequence of functors for motives with transfers
\begin{equation}\label{functors2}\Sh^\tr_\et(S)\to D(\Sh_\et(\Sm^{\tr}))\to\DMeff(S)\stackrel{L\Sigma^\infty_\tr}\to\DM(S)\ .\end{equation}

The functors $\gamma^*$ and $\gamma_*$ between abelian categories of sheaves induce Quillen adjunctions at the level of the model categories on motivic complexes and spectra (see \cite[2.2.6]{CDet} for the case of the descent model structures): 
\[ L\gamma^*:\DAeff(S)\rightleftarrows \DMeff(S):\gamma_*\ .\]
\[ L\gamma^*:\DA(S)\rightleftarrows  \DM(S):R\gamma_*\ .\]

The notation reflects:
\begin{lemma}
The functor $\gamma^*$ on motivic complexes preserves $\A^1$-equivalences.
\end{lemma}
\begin{proof}
The proof of \cite[Lemme 2.111]{Ayoub_Galois_1} also works in our setting.
\end{proof}
 
\begin{theorem}{\cite[Theorem 16.2.18]{CD} and \cite[Theorem 16.1.4]{CD}} \label{comp_transfers}
Let $S$ be noetherian, finite dimensional, excellent and geometrically unibranch. Then the functor 
\[ L\gamma^*:\DA(S)\rightleftarrows  \DM(S):\gamma_*\ .\]
is an equivalence of categories.
\end{theorem}
\begin{remark}
\begin{enumerate}
\item There is an integral version of this in \cite[Th\'eor\`eme~B.1]{Ayoubet} (there the functor $L\gamma^*$ is denoted $La_{\tr}$.)
\item In the effective case, the optimal comparison result is not known. For the state of the art in the effective case, see \cite[Annexe B.]{Ayoub_Galois_1} (for a base of characteristic 0) and \cite[Theorem 3.19]{Vezzani} (a weaker version over a perfect base).
\end{enumerate}
\end{remark}

\begin{definition} Let $X$ be a smooth $S$-scheme.
\begin{enumerate}
\item The {\em (relative, homological) effective motive with transfers} $M_S^\eff(X)^\tr$ is defined as the image of $\Q(X)^\tr$ in $\DMeff(S)$.
\item The {\em motive with transfers} $M_S(X)^\tr$ is defined as $L\Sigma^\infty_{\tr} M_S^\eff(X)^\tr$ in $\DM(S)$.
\end{enumerate}
\end{definition}

We collect some useful computational properties:
\begin{lemma}
\label{lemma_transfers}
The following statements hold:
\begin{enumerate}
\item \label{transfers_kunn}  All functors in the sequence (\ref{functors2}) are monoidal.
\item \label{suspension_transfers}There is a functorial $2$-isomorphism $L\gamma^*\Sigma^\infty\simeq \Sigma^\infty_\tr L\gamma^*$.
\item \label{transfer on motives} For all $X\in \Sm$, we have  $L\gamma^* M_S(X)\simeq M_S(X)^\tr$.
\item \label{transfers_tensor} In both the effective and non-effective case, the functor $L\gamma^*$ is monoidal.
\end{enumerate}
\end{lemma}
\begin{proof}
Assertion (\ref{transfers_kunn}) follows from \cite[Proposition 10.1.2]{CD} and the general machinery of monoidal $\mathcal{P}$-fibered categories of loc. cit., together with the fact that the tensor product of sheaves with transfers with $\Q$-coefficients is exact.

Assertion (\ref{suspension_transfers}) is contained in \cite[Corollary 10.3.11]{CD} and \cite[5.3.28]{CD}.

For Assertion (\ref{transfer on motives}), the motivic spectrum $\Sigma^\infty\Q(X)$ is cofibrant for the model structure on motivic spectra by \ref{rem:repr_cofibrant}. Moreover the functor $\gamma^*$ on motivic spectra is left Quillen (this is also in \cite[Corollary 10.3.11]{CD} and \cite[5.3.28]{CD}) and commutes with suspensions by (\ref{suspension_transfers}). To conclude we apply Lemma \ref{lemma with without} (\ref{transfers on X}).

Assertion (\ref{transfers_tensor}) is a formal consequence of Lemma \ref{lemma with without} (\ref{transfers monoidal}).

\end{proof}

\section{$\qfh$-descent for smooth group schemes}\label{app_qfh}

Let $S$ be a noetherian excellent scheme. In this section we study the qfh sheafification of smooth commutative group schemes over $S$.

\begin{remark}\label{qfh_normal}
Let $\Nor$ be the full subcategory of $S$-schemes of finite type which are normal. Let $T$ be an $S$-scheme of finite type. Then there is a normal $S$-scheme $T'$ and a finite surjective morphism (hence a qfh cover) $T'\to T$: take the normalisation of the reduction of the union of irreducible components of $T$. In particular the subsite $\Nor$ equipped with the induced qfh topology is dense in $\Sch$, and by \cite[Expos\'e III Th\'eor\`eme 4.1]{SGA4_I}  the inclusion induces an equivalence of topoi.
\end{remark}

Let $G/S$ a smooth group scheme and $\ul{G}$ be the \'etale sheaf on $\Sch/S$ defined by $G$.
The main result of this appendix is the following.

\begin{prop}\label{prop:qfh-descent}
Let $G$ be a smooth commutative group scheme over $S$. 
Then $\ul{G}_\Q|_\Nor$ is a qfh sheaf, and the natural morphism
\[
\ul{G}_\Q\rightarrow (\ul{G}_\Q)_\qfh
\]
induces isomorphisms when evaluated on normal schemes.
\end{prop}
The second statement is a reformulation of the first by Remark \ref{qfh_normal}.
The proof of the fact that $\ul{G}_\Q|_\Nor$ is a qfh sheaf will take the rest of this appendix.

\begin{lemma}\label{lemma:qfh-separation}
Let $f:X\to Y$ be a dominant morphism with $Y$ reduced. Then 
$\ul{G}(Y)\to\ul{G}(X)$ is injective. In particular the presheaf $\ul{G}|_\Nor$ is separated
with respect to the $\qfh$-topology.
\end{lemma}
\begin{proof}
  The morphism $f$ is schematically dominant since $Y$ is reduced \cite[Proposition 11.10.4]{EGAIV_3}. Let $g:Y\rightarrow G$ such that $g\circ f=0\in \ul{G}(X)$. Since the scheme $G$ is separated, we can now apply \cite[Proposition 11.10.1 d)]{EGAIV_3} to $g$ and the constant zero morphism $X\rightarrow G$ which shows that $g=0\in \ul{G}(Y)$. The $\qfh$-separation property then follows from \cite[Proposition 3.1.4]{VHS}. 
\end{proof}

\begin{lemma}\label{lemma:finite}
Let $\pi:U\to X$ be a finite surjective morphism of normal irreducible
schemes. Then $\ul{G}_\Q$ satisfies the sheaf condition for the qfh cover~$\pi$.
\end{lemma}
\begin{proof}
Let $V$ be the normalisation of the reduction of $U\times_X U$. We have to show that the sequence
\[ 0\to G(X)\tensor\Q\to G(U)\tensor\Q\to  G(V)\tensor\Q\]
is exact (the map $G(U)\tensor\Q\to  G(V)\tensor\Q$ being induced by the difference of the two projections $V \to U$). 

As $X$ and $U$ are normal, there exists $\pi':Y\to U$ finite surjective such that $\pi'\pi$ factors as $\pi_s\pi_i$ with $\pi_s$, $\pi_i$ finite surjective, $\pi_s$ generically Galois and $\pi_i$ generically purely inseparable. Because of Lemma \ref{lemma:qfh-separation}, a diagram chase shows that the sheaf condition for $\pi$ is implied by the sheaf condition for $\pi_s$ and $\pi_i$ separately. In other words we can assume $\pi$ to be either generically Galois or generically purely inseparable.

We first assume $\pi$ to be generically Galois with Galois group $\Gamma$. Note that $U$, being normal, is the normalisation of $X$ in $\kappa(U)$ so the action of $\Gamma$ on the generic fibre extends to $U$ by functoriality of the normalisation . Moreover, the monomorphism of coherent sheaves $\mathcal{O}_X\to (\pi_*\mathcal{O}_Y)^\Gamma$ is an isomorphism, as can be checked on affine charts using the generic Galois property and normality. Hence by \cite[Expos\'e V Proposition 1.3]{SGA1}, $X$ is the categorical quotient of $U$ by $\Gamma$. In particular we have:
\[ G(X)=G(U)^\Gamma\subset G(U)\ .\]

Let $x$ be in the kernel of $G(U)\to G(V)$. Then by Lemma \ref{lemma:qfh-separation} it is also in
the kernel of $G(k(U))\to G(k(U)\tensor_{k(X)}k(U))$ (recall that $k(U)\tensor_{k(X)}k(U)$
is a product of fields, hence normal). By \'etale descent we have
$x\in G(k(U))^\Gamma$, in particular, our element is $\Gamma$-invariant, hence
in $G(X)$ by the above.

Consider now the case $\pi$ generically inseparable. By Zariski's connectedness theorem, a finite surjective morphism to a normal scheme which is generically purely inseparable is purely inseparable. In particular the diagonal morphism of $\pi$ is a surjective closed immersion, hence after reduction is an isomorphism. Hence the map $G(U)\tensor\Q \to G(V)\tensor\Q$ is the zero map (the identity minus the identity) and we have to show that $\pi^*\tensor\Q:G(X)\tensor\Q\to G(U)\tensor\Q$ is surjective. For this we will reduce to the case where $\pi$ is a relative Frobenius. 

Let us recall some notations. For $p$ prime, $q$ a prime power and an $S$-scheme $Z$, we write $Z^{(q)}$ for the base change of $Z$ along the absolute Frobenius of $S$, and $\Frob^{(q)}_{Z/S}:Z\to Z^{(q)}$ for the relative Frobenius.

 We can assume that $\pi$ is not an isomorphism. Let then $p>0$ be the generic characteristic of $X$. Since $X$ is irreducible, it is an $\mathbb{F}_p$-scheme. By \cite[Proposition 6.6]{Kollar}, there exists a power $q$ of $p$ and a morphism $\pi^{-q}:X\to U^{(q)}$ such that $\pi^{-q}\pi=\Frob^{(q)}_{U/S}$.
 Hence we are reduced to the case  Frobenius case as surjectivity of $(\Frob^{(q)}_{U/S})^*\tensor\Q$ implies surjectivity of  $\pi^*\tensor\Q$.

Let $f\in G(U)$. By functoriality of Frobenius, we get a morphism 
\[ f^{(q)}\in G^{(q)}(U^{(q)})\]
such that $f^{(q)}\Frob^{(q)}_{U/S}=\Frob^{(q)}_{G/S}f$. The commutative $S$-group scheme $G$ is flat, so by \cite[Expos\'e VII 4.3]{SGA3} there exists a Verschiebung morphism $\Ver^{(q)}_{G/S}:G^{(q)}\to G$ such that $\Ver^{(q)}_{G/S}\circ \Frob^{(q)}_{G/S}=q\id_{G}$. Put 
\[ g=\Ver^{(q)}_{G/S}f^{(q)}\in G(U^{(q)})\ .\]
Then $\pi^*(g)=pf$. We conclude that $\pi^*\tensor\Q$ is surjective.
\end{proof}

\begin{proof}[Proof of Proposition \ref{prop:qfh-descent}.]
Let $X\in\Nor$ and $\{p_i:Y_i\to X\}_{i\in I}$ a qfh-cover in $\Nor$. We have to check the sheaf condition for $\ul{G}_\Q$. 

The presheaf $\ul{G}_\Q$ satisfyies the sheaf condition for the covering of a scheme by the union of its connected components so that we can assume $X$ to be connected, hence integral. By \cite[Lemma 10.3]{SV1}, the cover $\{p_i\}$ admits a refinement to a cover $\{Z_i\to Z \to X\}_{j\in J}$ where $Z\to X$ is finite surjective and $\{Z_i\to Z\}$ is a Zariski cover of $Z$. Because of Lemma \ref{lemma:qfh-separation}, a diagram chase shows that it is implied by the sheaf condition for $Z\to X$ and $\{Z_i\to Z\}$ separately.

By Lemma \ref{lemma:finite}, the presheaf $\ul{G}_\Q$ satisfies the sheaf condition for the morphism $Z\to X$. As $\ul{G}_\Q$ is a Zariski-sheaf, the sheaf condition
is also satisfied for the cover $\{Z_i\}$ of $Z$. This concludes the proof.
\end{proof}

\section{Eilenberg-MacLane complexes for sheaves}\label{app_resolution}

The main object of study in this paper is the sheaf $\ul{G}$ on $\Sm/S$ associated to a commutative algebraic group space $G/S$. In the course of the proof of the main theorem, we need to compute several left derived functors applied to the object $M^\eff_1(G)$ it represents in $\DMeff(S)$. For this, we need a cofibrant resolution of $\ul{G}\otimes\Q$. We use a construction of Breen \cite{Breen} that is based on the work of Eilenberg and Maclane \cite{EM1}. 

The following theorem is the main result of this appendix.

\begin{theorem}
  \label{theorem:resolution}
Let $(\mathcal{S},\tau)$ be a Grothendieck site. 
We denote $\Z(-)$ the functor ``free abelian group sheaf'' (the sheafification of the sectionwise free abelian group functor).

There is a functor:
\[
A:\Sh_\tau(S,\Z)\ra \mathbf{Cpl}_{\geq 0}\Sh_\tau(S,\Z)
\] 
(where by $\mathbf{Cpl}_{\geq 0}$ we mean homological complexes in non-negative degrees) together with a natural transformation
\[
r:A\rightarrow (-)[0]
\]
satisfying the following properties:
\begin{enumerate}
\item \label{explicit_form} For all $\Gh\in \Sh_\tau(S,\Z)$ and $i\geq 0$, the sheaf $A(\Gh)_i$ is of the form $\bigoplus_{j=0}^{d(i)}\Z(\Gh^{a(i,j)})$ for some $d(i),a(i,j)\in\N$. 
\item \label{lifting} There is a natural transformation $\tilde{a}:\Z(-)[0]\ra A$ which lifts the addition map $a:\Z(-)\rightarrow \id$; that is, one has $a[0]=r\tilde{a}$.
\item \label{site_functoriality} The functor $A$ and the transformations $r$ and $\tilde{a}$ are compatible with pullbacks by morphisms of sites. 
\item \label{Q-resolution} The map $r\otimes \Q$ is a quasi-isomorphism.
\end{enumerate}
\end{theorem}
We give a sketch of the proof of Theorem \ref{theorem:resolution} based on \cite{Breen} and \cite{EM1}. 
\subsection*{Construction of $A(\Gh)$}
The construction of the chain complex $A(\Gh)$ is done in four steps.
\begin{enumerate}[(i)]
\item The chain complex $A(\Gh,0)$ is defined as $\Z(\Gh)[0]$, equipped with its natural ring structure.
\item To define recursively $A(\Gh,n+1)$ from $A(\Gh,n)$, one applies the \emph{normalized bar construction} $B_N$ defined in \cite[Chapter II]{EM1} (see also Remark \ref{remarkbar})  
\[ A(\Gh,n+1)=B_N(A(\Gh,n)) .\]

\item For all $n\geq 1$, we define $A^n(G)$ by taking the canonical truncation (intuitively passing to the reduced homology) and shifting:
\[ A^n(\Gh)=(\tau_{>0}A(\Gh,n))[-n]\ . \]

\item The bar construction comes together with a functorial suspension map 
\[ S:A(\Gh,n)[1]\rightarrow A(\Gh,n+1)\ .\]
For all $n\geq 1$, up to a change of sign in the differential, the suspension $S$ induces a morphism of complexes \[S:A^n(\Gh)\rightarrow A^{n+1}(\Gh) \ .\] The maps are well-defined by the vanishing properties in Lemma~\ref{vanishing A^n} below, and they are inclusions. Define the chain complex $A(\Gh)$ as 
\[ A(\Gh)=\bigcup_{n\geq 1}A^n(\Gh)\ .\]
\end{enumerate}

\begin{lemma}[\cite{Breen}]\label{vanishing A^n} Let $A(\Gh,n)_j$ be the $j$-th term of the chain complex $A(\Gh,n)$; then
\[ A(\Gh,n)_j=\begin{cases}\Z& n\geq 1,j=0,\\
                                 0& n\geq 1,0<j<n.\\
                                 \Z(G) & n\geq 1, j=n
		\end{cases}\]
		
		Moreover, the map $A(G,1)_1\rightarrow A(G,1)_0$ is zero. In particular, for all $n\geq 1$ the chain complex $A^n(\Gh)$ is concentrated in non-negative degrees and can also be described via the stupid truncation
\[ A^n(\Gh)=(\sigma_{>0}A(\Gh,n)[-n]\ ,\]
which is still concentrated in non-negative degrees.
\end{lemma}

\begin{remark}\label{remarkbar}
Recall that the bar construction $B_N$ takes an augmented differential graded chain algebra and produces another; intuitively, if $D$ is an augmented dg-algebra computing the homology of a space $X$, the dg-algebra $B_N(D)$ should be thought of as an algebraic model for the homology of the loop space of $X$. Note that since we work with sheaves, the correct construction is to apply the bar construction sectionwise and then sheafify. 
\end{remark}

\subsection*{Construction of $r$ and end of the proof}
We first recall a qualitative version of Cartan's computation of the stable homology of $A(\Gh,n)$, as extended to the sheaf case by Breen (see \cite[Theorem 3]{Breen}): 
\begin{lemma}\label{lemma:Cartan}
Let $n\geq 1$. Then:
\begin{equation}\label{homology_computation_shifted}
  \Hh_{i}(A^n(\Gh))\simeq \begin{cases} 0&i<0, \\ \Gh& i=0,n\geq 1, \\ \text{is a torsion sheaf}& 0<i<n.\end{cases}
\end{equation}
Moreover, the isomorphisms $\Hh_0(A^n(\Gh))\simeq \Gh$ for $n\geq 1$ can be chosen such that
\begin{enumerate}
\item they are compatible with the suspension maps, and
\item the composition of maps of complexes 
  \[
  \Z(\Gh)[0]\rightarrow A^1(\Gh)\rightarrow \Hh_0(A^1(\Gh))\simeq \Gh[0]
  \]
  is the addition map $a_\Gh$ (here the map $\Z(\Gh)[0]\rightarrow A^1(\Gh)$ is the inclusion of the term in degree $0$). 
\end{enumerate}
\end{lemma}
\begin{proof}
  The computation of $\Hh_i(A(\Gh,n))$ is contained in \cite[Theorem 3]{Breen}. It reads
\begin{equation}\label{homology_computation}
\Hh_{i}(A(\Gh,n))\simeq \begin{cases}0&0< i< n, \\ \Gh,& i=n,\\ \text{is a torsion sheaf},& n<i<2n,\end{cases}
\end{equation}
The computation of $\Hh_i(A^n(\Gh))$ follows from this and the vanishing properties of Lemma \ref{vanishing A^n}. Compatibility with suspensions is a consequence of \cite[(eq 1.16) p.22]{Breen}. The second assertion follows by a small computation from the following description of the first terms of $A(\Gh,1)$ (see \cite[p.19]{Breen}):
\[
\xymatrix@R=0.3cm{
\ldots \ar[r] & \Z(\Gh\times \Gh) \ar[r] & \Z(\Gh) \ar[r]^0 & \Z \\
&  [g,h] \ar@{|->}[r] & [g]+[h]-[g+h] & &
}
\] 
This finishes the proof of the Lemma.
\end{proof}

\begin{definition}We define the map
\[ r:A(\Gh)=\bigcup_{n\geq 1}A^n(\Gh)\to \Gh\]
via the compatible system of 
\[ A^n(\Gh)\to \Hh_0(A^n(\Gh))\isom \Gh\]
of Lemma \ref{lemma:Cartan}.
\end{definition}
\begin{proof}[Proof of the Theorem \ref{theorem:resolution}]

Property (\ref{explicit_form}) follows from \cite[eq. (3.8) and (3.9) p.31]{Breen}.

The morphism $r$ satisfies property~(\ref{lifting}) by construction and Lemma 
\ref{lemma:Cartan}. It induces an isomorphism $\Hh_0(r):\Hh_0(A(\Gh))\simeq \Gh$. The construction is sufficiently canonical to make property (\ref{site_functoriality}) clear. 

Finally as the computation of homology commutes with filtered colimits and $\Q$ is flat over $\Z$, we have for all $i>0$
\[
\Hh_i(A(\Gh))\otimes\Q\simeq \mathrm{colim}_{n\in\N} (\Hh_i(A^n(\Gh))\otimes \Q).
\]
By Equation \ref{homology_computation_shifted}, the last term vanishes for $n>i$. This proves property (\ref{Q-resolution}) and concludes the proof.
\end{proof}

\end{appendix}
\bibliographystyle{alpha}	
\newcommand{\etalchar}[1]{$^{#1}$}
\def\cprime{$'$}

\end{document}